\documentclass[pdflatex,sn-mathphys,Numbered]{sn-jnl}%

\usepackage{tikz}
\usepackage{graphicx}%
\usepackage{multirow}%
\usepackage{amsmath,amssymb,amsfonts}%
\usepackage{amsthm}%
\usepackage{mathrsfs}%
\usepackage[title]{appendix}%
\usepackage{xcolor}%
\usepackage{textcomp}%
\usepackage{manyfoot}%
\usepackage{booktabs}%
\usepackage{algorithm}%
\usepackage{algorithmicx}%
\usepackage{algpseudocode}%
\usepackage{listings}%

\newtheorem{theorem}{Theorem}
\newtheorem{propo}[theorem]{Proposition}
\newtheorem{cor}[theorem]{Corollary}
\newtheorem{lem}[theorem]{Lemma}
\theoremstyle{definition}

\newtheorem{conj}{Conjecture}

\newtheorem{rem}{Remark}
\newtheorem{hyp}{Hypothesis}
\newtheorem{Assumption}{Assumption}

\begin{document}

\title{
Dispersal-induced growth or decay in a time-periodic environment. 
The case of reducible migration matrices.}

\author[1]{\fnm{Michel} \sur{Benaim}}\email{michel.benaim@unine.ch}

\author[2]{\fnm{Claude} \sur{Lobry}}\email{claude\_lobry@orange.fr}

\author*[3,5]{\fnm{Tewfik} \sur{Sari}}\email{tewfik.sari@inrae.fr}

\author[4]{\fnm{Edouard} \sur{Strickler}}\email{edouard.strickler@univ-lorraine.fr}

\affil[1]{\orgdiv{
Institut de Math\'ematiques}, 
\orgname{Universit\'e de Neuch\^atel}, 
\country{Switzerland}}

\affil[2]{\orgdiv{C.R.H.I}, 
\orgname{Universit\'e Nice Sophia Antipolis}, 
\country{France}}

\affil[3]{\orgdiv{ITAP}, 
\orgname{Univ Montpellier, INRAE, Institut Agro}, \city{Montpellier}, 
\country{France}}

\affil[4]{\orgdiv{IECL}, 
\orgname{Universit\'e de Lorraine, CNRS, Inria},  
\city{Nancy}, %
\country{France}}

\affil[5]{\orgdiv{GreenOwl project team}, \orgname{Universit\'e C\^ote d'Azur, Inria, INRAE, CNRS-Sorbonne Universit\'e (LOV)}, 
\city{Valbonne},
\country{France}}

\abstract{This paper is a follow-up to a previous work where we considered  populations with time-varying growth rates living in patches and irreducible migration matrix between the patches. Each population, when isolated, would become extinct. Dispersal-induced growth (DIG) occurs when the populations are able to persist and grow exponentially when dispersal among the populations is present. 
In this paper, we consider the situation where the migration matrix is not necessarily irreducible. 
We provide a mathematical analysis of the DIG phenomenon, in the context of a deterministic model with periodic variation of growth rates and  migration. Our results apply in the case, 
important for applications, where there is migration in one direction in one season and in the other direction in another season. 
We also consider dispersal-induced decay (DID), where each population, when isolated, grows exponentially, while populations die out when dispersal between populations is present.}




\maketitle

\tableofcontents

\section{Introduction}
\label{intro}

We considered in  \cite{BLSS} the model of 
populations of sizes $x_i(t)$ ($1\leq i\leq n$), inhabiting $n$ patches, and subject to time-periodic local growth rates $r_i(t)$ ($1\leq i\leq n$), and migration $m\ell_{ij}(t)\geq 0$,  from patch $j$ to patch $i$,  where the parameter $m\geq 0$ measures the strength of migration, and 
the numbers 
$
\ell_{ij}(t)$, $i\neq j$,
 encode the topology of the dispersal network and the relative rates of dispersal among different patches: At time $t$, there is a migration from patch $j$ to patch $i$ if and only if $\ell_{ij}(t) > 0$. 
We then have the differential equations
\begin{equation}\label{eqBLSS}
\frac{dx_i}{dt}=r_i(t)x_i+m\sum_{j\neq i}\left(\ell_{ij}(t)x_j-\ell_{ij}(t)x_i\right),
\quad 1\leq i\leq n. 
\end{equation}

It is assumed that

\begin{hyp}\label{H1} The functions $r_i(t)$ and $\ell_{ij}(t)$ are $1$-periodic 
functions, which are piecewise continuous with a finite set of discontinuity points on each period. Moreover, they have left and right limits at the discontinuity points.
\end{hyp}

We proved, see \cite[Proposition 1]{BLSS} that in the irreducible migration case (i.e. for any $t$, any two patches are connected by migration, either directly, or through other patches), if $m>0$, any solution of (\ref{eqBLSS}) with $x_i(0)\geq 0$ for $1\leq i\leq n$ and $x_i(0)> 0$ for some $i$, satisfies $x_i(t)>0$ for all $i$ and all $t>0$. Moreover, the limits
\begin{equation}\label{Lambdai}
\Lambda[x_i]:=\lim_{t\to\infty}\frac{1}{t}\ln(x_i(t)),
\end{equation}
exist, are equal, and 
they do not depend on the initial condition. Their common value is called the growth rate of the system \eqref{eqBLSS}. When this property is satisfied, we say that the system \emph{admits a growth rate}, or \emph{the growth rate exists}.  

Let $T>0$. If we replace the time $t$ by $t/T$ in the right-hand side of \eqref{eqBLSS} we obtain the $T$-periodic linear system:
\begin{equation}\label{eq1}
\frac{dx_i}{dt}=r_i(t/T)x_i+m\sum_{j\neq i}\left(\ell_{ij}(t/T)x_j-\ell_{ij}(t/T)x_i\right),
\quad 1\leq i\leq n. 
\end{equation}
The common value of the $\Lambda[x_i]$ given by \eqref{Lambdai} is denoted $\Lambda(m,T)$ to recall its dependence on the parameters $m$ and $T$.  The main results in \cite{BLSS} are on the asymptotic properties of $\Lambda(m,T)$ when $T\to 0$ and $T\to\infty$.

A patch $i$ is called a  \emph{sink} when $\overline{r}_i=\int_0^1r_i(t)dt<0$, i.e. in the
absence of dispersion ($m=0$), the population goes to extinction. In the case where all patches are sinks, it is sometimes possible to find  $m>0$ and $T>0$ such that $\Lambda(m,T)>0$,  i.e. the population is exponentially growing on each patch.
Since it
is possible for populations in a set of patches, with dispersal among them, to persist and
grow despite the fact that all these patches are sinks, this phenomenon was called 
\emph{dispersal-induced growth} (DIG) by Katriel \cite{Katriel}, who considered the case of time independent and  symmetric  migration (i.e. $\ell_{ij}=\ell_{ji}$) and when the functions $r_i(t)$ are continuous. 

This surprising phenomenon of 
populations that can \emph{persist in an environment consisting of sink habitats only} was first studied by Jansen and Yoshimura \cite{Jansen} and has already been pointed by Holt \cite{Holt} on particular systems and called {\em inflation} \cite{Holt-et-al}.
For further details and complements on the mathematical modelling of this phenomenon and the biological motivations, the reader is referred to \cite{benaim,BLSS,Katriel} and the references therein.

The matrix
$L(t)$ whose off diagonal elements are $\ell_{ij}(t)$, $i\neq j$, and diagonal elements $\ell_{ii}(t)$ are given by
$
\ell_{ii}(t)=-\sum_{j\neq i} \ell_{ji}(t),$ $1\leq i\leq n,
$
is called the migration or dispersal matrix. 
Using the matrix $L(t)$,  the equations (\ref{eq1}) can be written as
\begin{equation}\label{eq3}
\frac{dx}{dt}=A(t/T){x},
\quad
\mbox{ where } 
A(\tau)=R(\tau)+mL(\tau),
\end{equation}
${x}=\left(
	x_1,\cdots,
	x_n\right)^\top$, and 
$R(\tau)={\rm diag}\left(
	r_1(\tau),\cdots,r_n(\tau)
\right).$
In addition to the assumptions that $L(\tau)$  has non-negative off diagonal elements, it is assumed in \cite{BLSS} that 
for all $t$, $L(t)$ is {\it{irreducible}}.
This assumption is certainly not realized in many real systems.
For instance on a two-patches system with two seasons, if there is migration in one direction in one season and in the other direction in
another season, then the matrix $L(\tau)$ would not be irreducible for the times at which
migration is in one direction only. The aim of this paper is to relax this condition on the irreducibility of the migration matrix and to replace it by  the following assumption. 
\begin{hyp}\label{H2}
The average migration matrix $\overline{L}=\int_0^1L(t)dt$ is irreducible.
\end{hyp}

A sufficient (but non necessary) condition ensuring that $\overline{L}$ is irreducible is that $L(t)$ is irreducible for some $t\in[0,1]$. Therefore, the results in this paper extend the results of \cite{Katriel} where the migration matrix is assumed to be time-independent and irreducible and the results of \cite{BLSS}, where $L(t)$ is assumed to be irreducible for all $t\in[0,1]$.

The results of \cite{Katriel} are based on the strict {monotonicity} of the function $T\mapsto\Lambda(m,T)$, which is true when the matrix $L$ is time independent and symmetric. Indeed, it is strictly increasing, except in the case where all the $r_i$'s are equals (see  \cite[Lemma 2]{Katriel}). This result follows from general results of Liu et al. \cite{liu} on the principal eigenvalue of a periodic linear system. 
But, as we have shown in \cite{BLSS},  the 
{monotonicity} 
of $T\mapsto \Lambda(m,T)$ is no longer true if $L(\tau)$ is not constant. We conjectured in \cite{BLSS} that this function is strictly increasing in the non symmetric constant case. But this conjecture is not true, as it was shown by Monmarch\'e et al. \cite[Figure 4]{MSS}.

To prove the results in \cite{BLSS} we used classical methods of dynamical systems theory: the Perron-Frobenius theorem, the method of averaging and Tikhonov's theorem on singular perturbations. The proofs in the more general context of the present paper use the same tools. For example, the existence of the growth rate $\Lambda(m,T)$ follows from the Perron-Frobenius theorem applied to the monodromy matrix, which is nonnegative and irreducible since $\overline{L}$ is irreducible (see Lemma \ref{lem1}, below). 
The determination of the limits of $\Lambda(m,T)$ as $m$ tends to 0 or infinity, follows the same steps as the proofs in \cite{BLSS} and makes use of the method of averaging and Tikhonov's theorem on singular perturbations (see Appendix \ref{Proofs}). We discuss the DIG phenomenon in the more general context, where Hypothesis \ref{H2} is satisfied. In this paper, we also consider the case of populations in a set of patches, with dispersal among them, that go to extinction, despite the fact that all these patches are sources. We call this phenomenon 
\emph{dispersal-induced decay} (DID).

The paper is organized as follows. In Section \ref{Results} we give our main results on the existence of the growth rate and its limits as $m$ and/or $T$ tend to 0 or infinity (some proofs are postponed to Appendix \ref{Proofs}). 
In Section \ref{Numerical}, by means of examples of the cases with 2 and 3 patches and {piecewise constant} growth rates and migration, we  illustrate our principal results and show how the needed hypothesis can be verified. 
In Section \ref{Discussion} we discuss the results in more detail and propose some questions for further research.
In Appendix \ref{SingPert} we provide a statement of the theorem of Tikhonov on singular perturbations which is used to prove the asymptotic behavior of the growth rate when the period is large ($T\to \infty$), or the migration rate is large ($m\to \infty$).

\section{Results}\label{Results}

Throughout the paper, the  following notation is used:
If $u(\tau)$ is any 1-periodic object (number, vector, matrix...), we denote by $\overline{u}=\int_0^1u(\tau)d\tau$ its average on one period. We also use the following notations:
for $x\in\mathbb{R}^n$,
$x\geq 0$ means that for all $i$, $x_i\geq 0$, $x> 0$ means that $x\geq 0$ and $x\neq 0$, and $x\gg 0$ means that for all $i$, $x_i> 0$. 

Let $\mathcal{M}$ be the set of Metzler $n\times n$ matrices, i.e. having off diagonal nonnegative entries. Let 
$A:\mathbb{R}\to \mathcal{M},$
be a 1-periodic function.
We consider the linear 1-periodic system of differential equations
\begin{equation}
\label{eqA}
\frac{dx}{dt}=A(t)x,
\end{equation}
with initial condition $x(0)>0$, under the following assumptions:
\begin{Assumption}\label{i}
There exist $\tau_i$, $i=0\ldots N+1$ with $0=\tau_0<\tau_1< \ldots<\tau_N<\tau_{N+1}=1$, such that for  $k=0\ldots N$, ${A}|_{[\tau_k,\tau_{k+1})}$ are continuous functions, and
$\lim_{\tau\to\tau_{k+1}}A(\tau)$ exist. 
\end{Assumption}\label{ii}
\begin{Assumption}
The average matrix $\overline{A}=\int_0^1A(t)dt$ is irreducible.
\end{Assumption}

{Assumption \ref{i} implies
that} the
solutions of \eqref{eqA} are continuous and piecewise $\mathcal{C}^1$ functions satisfying  \eqref{eq1} except at the points of discontinuity of the function $A(t)$.

\subsection{The growth rate}\label{INNsec}
{The aim of this section is to show that under Assumptions \ref{i} and \ref{ii}, the system  \eqref{eqA} admits a growth rate.} 
Since the system \eqref{eqA} is a 1-periodic system, its study reduces to the study of its monodromy matrix $X(1)$, where $X(1)$ is the value at time $t=1$ of 
the fundamental matrix solution $X(t)$, i.e.  the solution of the matrix-valued differential equation
\begin{equation}\label{MVDE}
\frac{dX}{dt}=A(t)X,\qquad X(0)=I,
\end{equation}
where $I$ is the identity matrix. Since $A(t)$ is Metzler, for all $t>0$, $X(t)$ has non-negative entries. 

\begin{lem}\label{lem1}
Under Assumptions  \ref{i} and \ref{ii}, 
for all $t$ sufficiently large, $X(t)$ has positive entries.
\end{lem}

\begin{proof}
The proof is similar to the proof of \cite[Lemma 6(i)]{BLSSarXiv}.
First observe that $X(t)$ has positive entries if and only if  $e^{rt} X(t)$ has positive entries, for all $r > 0.$
Therefore, replacing $A(\tau)$ by $A(\tau) + r Id$ for $r > \|A\|_{\infty},$\footnote{$\|A\|_{\infty} = \sup_{\tau \in [0,1]} \|A( \tau)\|$ is finite thanks to Assumption \ref{i}. } we can assume without loss of generality that $A(\tau)$ has nonnegative entries for all $\tau \in [0,1].$

Let $x(t) = X(t) x(0)$ with $x(0)>0.$
Suppose $x_i(0) > 0.$ Then $x_i(t) > 0$ because $\dot{x}_i(t) \geq A_{ii}(t) x_i(t) \geq 0.$
 By irreducibility of $\overline{A}$, for all $j \neq i$, there exists a sequence $i_0 = i, i_1, \ldots, i_p = j$ such that $\overline{A}_{i_k i_{k+1}} > 0$ for $k = 1,\ldots p.$ Since $ \tau \mapsto A(\tau)$ is 1-periodic, one has, for any integer $N \geq 1$,
 $$\frac{1}{N} \int_0^N A(\tau )d \tau=  \overline{A}.$$
 Therefore,  there exists a sequence $t_1 < t_2 < \ldots < t_p$ with  
 $$A_{i_k i_{k+1}}(t_k) > 0.$$
  By right continuity of $\tau \mapsto A(\tau),$ we also have $A_{i_k i_{k+1}}(t) > 0$ for $t_k \leq t \leq t_k + \varepsilon$ for some $\varepsilon > 0.$
  It follows that $\dot{x}_{i_1}(t) \geq A_{i_1,i}(t) x_1(t) > 0$ for all $t_1 \leq t \leq t_1 + \varepsilon.$
  Hence $x_{i_1}(t) > 0$ for all $t > t_1.$ Similarly $x_{i_2}(t) > 0$ for all $t > t_2$ and, by recursion, $x_j(t) > 0$ for all $t > t_p.$  In summary, we have shown that for all $i,j \in \{1, \ldots, n\}$ there exists a time $t_p$ depending on $i,j$, such that for all $t \geq t_p$ $x_j(t) > 0$ whenever $x_i(0) > 0.$ Hence, for all $t$ sufficiently large, provided $x_i(0) > 0$, one has $x_j(t) > 0$, which implies that the $i-th$ column of $X(t)$ has positive entries, and conclude the proof.
\end{proof}

The previous lemma insures that the monodromy matrix $X(1)$ is irreducible. Indeed, $X(1)^p=X(p)$ has positive entries for all sufficiently large integer $p$.
By the Perron-Frobenius theorem for irreducible non negative matrices \cite[Page 673]{MeyerBook}, $X(1)$ has a dominant eigenvalue (an eigenvalue of maximal modulus, called the \emph{Perron-Frobenius root}), which is positive. 
We denote it by  
$\mu$. There exists a unique vector, called \emph{Perron-Frobenius vector}, $\pi\gg 0$, such that $\sum_{i=1}^n\pi_i=1$ and $X(1)\pi=\mu\pi$.

The change of variables
$
\rho=\sum_{i=1}^nx_i,$ 
$\theta=\frac{x}{\rho},$
transforms the system \eqref{eqA} into
\begin{align}\label{eqrho}
\frac{d\rho}{dt}&=\langle A(t)\theta,
{\bf 1}\rangle \rho,
\\
\label{eqtheta}
\frac{d\theta}{dt}&=A(t)\theta- \langle A(t)\theta,
{\bf 1}\rangle \theta.
\end{align}
Here ${\bf 1}=(1,\ldots,1)^\top$ and $\langle x,{\bf 1}\rangle=\sum_{i=1}^nx_i$ is the usual Euclidean scalar product of vectors $x$ and ${\bf 1}$. 
The equation \eqref{eqtheta} is a differential equation on $\Delta$, the unit simplex $n-1$ simplex of $\mathbb{R}_+^n=\{x\in\mathbb{R}^n:x>0\}$, defined by
$$\Delta:=\left\{x\in\mathbb{R}_+^n:\textstyle\sum_{i=1}^nx_i=1\right\}.$$

\begin{theorem}\label{thm2}
Let Assumptions \ref{i} and \ref{ii} be satisfied. Let $\mu$ be the Perron-Frobenius root of the monodromy matrix $X(1)$ of \eqref{eqA}. Let 
$\pi\in\Delta$, be its Perron-Frobenius vector.
Let $\Lambda:=\ln(\mu)$. The solution  $\theta^*(t)$ of 
 \eqref{eqtheta}, such that $\theta^*(0)=\pi$, is 1-periodic, and globally asymptotically stable in $\Delta$. Moreover, if $x(t)$ is a solution of \eqref{eqA} such that $x(0)>0$, then $x(t)\gg0$ for all sufficiently large $t$ and, for all $i$,
\begin{equation} \label{Lambda=Lambda[rho]1}
 \lim_{t\to\infty}\frac{1}{t}\ln(x_i(t))=\int_0^1\langle A(\tau)\theta^*(\tau),{\bf 1}\rangle d\tau=\Lambda.
 \end{equation}
 \end{theorem}

\begin{proof}
The proof is given in Section \ref{ProofThm2}.
\end{proof}

\begin{cor}\label{Bornes}
The following inequalities are true:
$\sigma\leq \Lambda \leq \chi,$ where
\begin{equation}
\label{chi_minmax}
\sigma=\overline{\min_{1\leq i\leq n}\textstyle\sum_{j=1}^nA_{ji}},
\quad
\chi=\overline{\max_{1\leq i\leq n}\textstyle\sum_{j=1}^nA_{ji}}
\end{equation}
\end{cor}

\begin{proof}
For all $t\geq 0$ and $\theta\in\Delta$, 
$$
\min_{1\leq i\leq n}\sum_{j=1}^nA_{ji}(t)
\leq
\langle A(t)\theta,{\bf 1}\rangle
\leq
\max_{1\leq i\leq n}\sum_{j=1}^nA_{ji}(t)
$$
and the result follows from the integral representation \eqref{Lambda=Lambda[rho]1} of $\Lambda$.
\end{proof}

Let $T>0$. If we replace the time $t$ by $t/T$ in the right-hand side of \eqref{eqA} we obtain the $T$-periodic linear system: 
\begin{equation}\label{eqAT}
\frac{dx}{dt}=A(t/T){x},
\end{equation}
The change of variables
$$t/T=\tau,\quad y(\tau)=x(T\tau),\quad Y(\tau)=X(T\tau)$$
transforms \eqref{eqA} and the corresponding matrix-valued equation \eqref{MVDE} into the equations
\begin{equation}\label{A(tau)}
  \frac{dy}{d\tau}=TA(\tau)y,\qquad \frac{dY}{d\tau}=TA(\tau)Y.  
\end{equation}
These equations are 1-periodic. Therefore, Theorem \ref{thm2} has the following corollary.

\begin{cor}\label{cor2}
Let Assumptions \ref{i} and \ref{ii} be satisfied. Let $\mu(T)$ be the Perron-Frobenius root of the monodromy matrix $X(T)$ of \eqref{eqAT}. Let $\Lambda(T):=\frac{1}{T}\ln(\mu)$. If $x(t)$ is a solution of \eqref{eqAT} such that $x(0)>0$, then $x(t)\gg0$ for all sufficiently large $t$ and, for all $i$,
$$\lim_{t\to\infty}\frac{1}{t}\ln(x_i(t))=\Lambda(T).$$
 \end{cor}

\begin{proof}
Using Theorem \ref{thm2}, for any solution $y(\tau)$ of \eqref{A(tau)}, such that $y(0)>0$,  $y(\tau)\gg0$ for all sufficiently  large $\tau$ and
$$\lim_{\tau\to\infty}\frac{1}{\tau}\ln(y_i(\tau))=\ln(\mu(T)),$$
where $\mu(T)$ is the Perron-Frobenius root of the monodromy matrix $Y(1)=X(T)$.  Since
$$\lim_{\tau\to\infty}\frac{1}{\tau}\ln(y_i(\tau))=
\lim_{t\to\infty}\frac{1}{t/T}\ln(x_i(t))=T\lim_{t\to\infty}\frac{1}{t}\ln(x_i(t)),$$
we deduce that for all $i$, 
$\lim_{t\to\infty}\frac{1}{t}\ln(x_i(t))=\frac{1}{T}\ln(\mu(T))=:\Lambda(T).$
\end{proof}

\subsection{Fast and slow regimes}\label{HFL}

The aim of this section is to determine the limits of the growth rate $\Lambda(T):=\frac{1}{T}\ln(\mu(T))$  for small and large $T$.

\subsubsection{Fast regime}\label{FR}
Using Assumption \ref{ii}, by the Perron-Frobenius theorem \cite[Page 673]{MeyerBook}, applied to $\overline{A}+rI$, with 
$r\geq \|\overline{A}\|_\infty$, 
the spectral abscissa of $\overline{A}$, i.e. the largest real part of its eigenvalues, is a simple eigenvalue of $\overline{A}$ and is denoted $\lambda_{max}(\overline{A})$.

\begin{theorem}\label{Prop13}
Let Assumptions \ref{i} and \ref{ii} be satisfied.  The growth rate  of 
\eqref{A(tau)} satisfies 
\begin{equation}\label{T=0}
\lim_{T\to 0}\Lambda(T)=\lambda_{max}(\overline{A}).
\end{equation}
\end{theorem}
\begin{proof}
The proof is given in Section \ref{ProofProp13}.
\end{proof}

\subsubsection{Slow regime}\label{SR}

Since $A(\tau)$ is Metzler, the Perron theorem \cite[Page 670]{MeyerBook}, applied to ${A(\tau)}+rI$ for $r\geq \|A(\tau)\|_\infty$, asserts that the spectral abscissa $\lambda_{max}(A(\tau))$ is an eigenvalue of $A(\tau)$ and it admits a non negative eigenvector $v(\tau)\in \Delta$.

\begin{lem}\label{lemma6}
Let $\lambda$ be an eigenvalue of $A(\tau)$, having a nonnegative eigenvector $v\in\Delta$. 
The eigenvector $v$ is an equilibrium point of the autonomous differential equation on the simplex $\Delta$:
\begin{equation}\label{eqsimplextau}
\frac{d\theta}{dt}=A(\tau)\theta- \langle A(\tau)\theta,{\bf 1}\rangle\theta.
\end{equation}
In particular, the eigenvector $v(\tau)\in\Delta$, corresponding to $\lambda_{max}(A(\tau))$, is an equilibrium point \eqref{eqsimplextau}.
In this equation,
$\tau\in[0,1)$ is considered as a parameter.
\end{lem}

\begin{proof}
Since  $A(\tau)v=\lambda v$, we obtain 
\begin{align*}
A(\tau)v- \langle A(\tau)v,{\bf 1}\rangle v
=\lambda v-\lambda\langle v,{\bf 1}\rangle v=0,
\end{align*}
because $\langle v,{\bf 1}\rangle=1$. Therefore $v$ is an equilibrium of 
\eqref{eqsimplextau}.
\end{proof}

As in \cite{BLSS}, we use
Tikhonov's theorem on singular perturbations to determine the limit of $\Lambda(T)$ as $T\to\infty$. 
This theorem needs that the equilibrium of the fast equation is asymptotically stable and has a basin of attraction which is uniform, see the condition (SP2) in  Appendix \ref{SingPert}. 
Since
\eqref{eqsimplextau} is the fast equation corresponding to the singularly perturbed equation arising when $T\to\infty$, we need to assume that $v(\tau)$ is 
asymptotically stable for the differential equation \eqref{eqsimplextau} and has a basin of attraction which is uniform with respect to the parameter $\tau\in[0,1)$. 
More precisely, we make the following assumption.

\begin{hyp}\label{H3}
We assume that \\
(H3.1) For each $\tau\in[0,1)$, $v(\tau)$ is an asymptotically stable equilibrium of \eqref{eqsimplextau}.\\
(H3.2) 
There exists $\delta>0$ such that, for each $\tau\in[0,1)$, the $\delta$-radius ball in $\Delta$, centred at $v(\tau)$, is included in the basin of attraction of $v(\tau)$.\\
(H3.3)
At the points $\tau_k$ of discontinuity, the basin of attraction of $v(\tau_k)$ contains 
the right limit $v(\tau_{k}-0)=\lim_{\tau\to \tau_k,\tau<\tau_k}v(\tau)$. 
\end{hyp}
\begin{rem}
Hypothesis \ref{H3} is satisfied if, for all $\tau\in[0,1)$, $v(\tau)$ is GAS in $\Delta$ for \eqref{eqsimplextau}. This is the case, in particular, when $A(\tau)$ is irreducible for all $\tau\in[0,1)$. Indeed, in this case, we have $v(\tau)\gg 0$, and it is GAS in $\Delta$, see \cite[Proposition 24]{BLSS}. However, if the matrix $A(\tau)$ is reducible,  $\lambda_{max}(A(\tau))$ is not necessarily simple (i.e. of algebraic multiplicity 1), we do not have $v(\tau)\gg0$ and there are possibly other nonnegative
eigenvectors for $A(\tau)$, corresponding to other eigenvalues. 
\end{rem}

\begin{rem}\label{RemH3}
In the piecewise constant case, Hypothesis \ref{H3} is satisfied if and only if for each interval 
$[\tau_k,\tau_{k+1})$, $v_k:=v(\tau_k)$ is asymptotically stable and its basin of attraction contains 
$v_{k-1}:=v(\tau_{k-1})$ .
\end{rem}

In Section \ref{Numerical}  we present several examples of how Hyposthesis \ref{H3} can be verified, see Section \ref{TPC} and Fig. \ref{figure10}(b), as well as cases where it is not, see Figs. \ref{figure7} and \ref{figure9}(b), where the condition H3.3 is not satisfied and Fig. \ref{figure10}(c), where the condition H3.1 is not satisfied. Indeed,
the equilibrium $v(\tau )$ is not always asymptotically stable: if $v(\tau )$ is not an isolated equilibrium, which can happen when $\lambda_{max}(A(\tau))$ is not a simple eigenvalue of $A(\tau)$, then $v(\tau )$ is not asymptotically stable. For an illustrative example, see Fig. \ref{figure10}(c).

\begin{theorem}\label{Prop14}
Let Assumptions \ref{i} and \ref{ii} and Hypothesis \ref{H3} be satisfied.  The growth rate  of 
\eqref{A(tau)} satisfies 
\begin{equation}\label{T=infini}
\lim_{T\to \infty}\Lambda(T)=\overline{\lambda_{max}(A)}
 \end{equation}
\end{theorem}
\begin{proof}
The proof is given in Section \ref{ProofProp14}.
\end{proof}

\subsection{Low and high migration rate}\label{FM}
We now consider the special case of \eqref{eq3}, where  $A(\tau)=R(\tau)+mL(\tau)$. Note that if Hypotheses  \ref{H1} and \ref{H2} are satisfied,  then Assumptions  \ref{i} and \ref{ii} on the system \eqref{eqA} are satisfied. Thus, Theorem \ref{thm2} asserts that the system \eqref{eq3} has a growth rate given by
$$\Lambda(m,T)=\frac{1}{T}\ln(\mu(m,T)),$$
where $\mu(m,T)$ is the Perron-Frobenius root of the monodromy matrix $X(m,T)$ of \eqref{eq3}. The aim of this section is to determine the limits of $\Lambda(m,T)$  for small and large $m$. 

\subsubsection{Low migration rate}

\begin{propo}\label{Prop12m=0}
Assume that Hypotheses \ref{H1} and \ref{H2} are satisfied. For all $T>0$, the growth rate $\Lambda(m,T)$ of 
\eqref{eq3} satisfies   
\begin{equation}\label{m=0}
\Lambda(0,T):=\lim_{m\to 0}\Lambda(m,T)=\max_i\overline{r}_i.
\end{equation}
\end{propo} 
\begin{proof}
The proof is the same as the proof of \cite[Eq. (14)]{BLSS}. Indeed, the proof in \cite{BLSS} only uses  the continuous dependence of the solutions of \eqref{eq3} on the parameter $m$. For the details, we refer the reader to  \cite[Section 5.5]{BLSS}. 
\end{proof}  

\subsubsection{High migration rate} 

We use the following result.
\begin{propo}\label{Lreducible}
Let $L$ be a matrix whose columns sum to 0. Then $0$ is an eigenvalue of $L$.
Assume that $L$ is Metzler (the off diagonal elements of $L$ are nonnegative). Let $k\geq 1$ be the algebraic multiplicity of its  eigenvalue 0. The following properties are true
\begin{enumerate}
\item All other eigenvalues of $L$ are of negative real part. 
\item The geometric multiplicity of 0 is equal to $k$, i.e. $L$ admits $k$ linearly independent eigenvectors.
\item The  $k$ linearly independent eigenvectors of $L$ can be chosen in the simplex $\Delta$.
\end{enumerate}
If $L$ is irreducible then $k=1$.
\end{propo}
\begin{proof}
The proof is given in Section \ref{ProofLreducible}.
\end{proof}

Since the columns of $L(\tau)$ sum to 0, 0 is an eigenvalue of $L(\tau)$ and it admits a non negative eigenvector $p(\tau)\in \Delta$. 
Since
$\langle L(\tau)\theta,{\bf 1}\rangle=0$,
the differential equation \eqref{eqsimplextau} on the simplex $\Delta$, corresponding to the linear equation $\frac{dx}{dt}=L(\tau)x$, is written
\begin{equation}\label{eqsimplextauL}
\frac{d\theta}{dt}=L(\tau)\theta,
\end{equation}
where $\tau\in[0,1)$ is considered as a parameter.  Note that $p(\tau)$ is an equilibrium point of the differential equation \eqref{eqsimplextauL}, because $L(\tau)p(\tau)=0$ and  $p(\tau)\in \Delta$.

As for the limit of $\Lambda(m,T)$ as $T\to\infty$, we use
Tikhonov's theorem on singular perturbations to determine the limit of $\Lambda(m,T)$ as $m\to\infty$. 
Since
\eqref{eqsimplextauL} is the fast equation corresponding to the singularly perturbed equation arising when $m\to\infty$, we need to assume that for each $\tau\in[0,1)$, $p(\tau)$ is 
asymptotically stable for the differential equation \eqref{eqsimplextauL} and has a basin of attraction which is uniform with respect to the parameter $\tau\in[0,1)$, see the condition (SP2) in  Appendix \ref{SingPert}. What we mean by the uniformity of the basin of attraction of $p(\tau)$ is the same  Hypothesis \ref{H3}, where $v(\tau)$ should be replaced by $p(\tau)$.  
Unlike \eqref{eqsimplextau}, where the equation can have other equilibria on the $\Delta$ simplex, corresponding to other eigenvalues of $A(\tau)$, as shown in Lemma~\ref{lemma6}, the novelty with \eqref{eqsimplextauL} is that there are no other equilibria on $\Delta$ apart from the eigenvectors corresponding to eigenvalue 0.  
 Indeed, \eqref{eqsimplextauL} is a compartmental system \cite{Jacquez} and its properties are well known. In addition to the properties given in Proposition \ref{Lreducible} we have the follwoing result.

\begin{lem} If $\lambda$ is an eigenvalue of $L(\tau)$ and $v$ a corresponding eigenvector then $\lambda=0$ or $\langle v,{\bf 1}\rangle=0$. Therefore, there is no eigenvector on $\Delta$ apart from eigenvectors corresponding to eigenvalue 0.  
 \end{lem}
 \begin{proof}
  Since the columns of $L(\tau)$ sum to 0, we have
$\langle L(\tau)v,{\bf 1}\rangle=0$. From
$L(\tau)v=\lambda v$ we deduce 
$\lambda\langle v,{\bf 1}\rangle=0$. Therefore, 
$\lambda=0$ or $\langle v,{\bf 1}\rangle=0$.
 \end{proof}
 
 This particularity of the system \eqref{eqsimplextauL} simplifies the hypotheses needed on $p(\tau)$. Indeed, we see that Hypothesis \ref{H3}, where $v(\tau)$ is replaced by $p(\tau)$ and equation \eqref{eqsimplextau} is replaced by \eqref{eqsimplextauL}, is satisfied, if and only if $L(\tau)$ satisfies the following hypothesis.

\begin{hyp}\label{H4}
We assume that for each $\tau\in[0,1)$, 0 is a simple eigenvalue of $L(\tau)$.
\end{hyp}

Indeed, we have the following result.
\begin{lem}\label{lemma10}
If Hypothesis \ref{H4} is satisfied then $p(\tau)$ is GAS in $\Delta$ for  \eqref{eqsimplextauL}.
If Hypothesis \ref{H4} is not satisfied then $p(\tau)$ is not asymptotically stable for  \eqref{eqsimplextauL}.
\end{lem}

\begin{proof}
The system \eqref{eqsimplextauL}, considered as a system on $\mathbb{R}^n$, not only on the simplex $\Delta$, is a singular linear compartmental system, with constant coefficients, and no input. Therefore, it falls under case (b) of \cite[Eq. (14)]{Jacquez} and the trajectories of \eqref{eqsimplextauL} lie on lower dimensional hyperplanes $\sum_{i=1}^n\eta_i={\rm constant}$ of $\mathbb{R}_+^n$ and tend exponentially to the zero set of $L$ on each of these hyperplanes. Thus, if $0$ is a simple eigenvalue of $L(\tau)$, then $p(\tau)$ is GAS in $\Delta$ for \eqref{eqsimplextauL}, so that all conditions in  
Hypothesis \ref{H3} are satisfied.

Conversely, if 0 is an eigenvalue of $L(\tau)$ of multiplicity $k$, using Proposition \ref{Lreducible}, there are always $k$ independent eigenvectors in the simplex $\Delta$ for the eigenvalue 0. Therefore, if $k>1$, the system \eqref{eqsimplextauL} has a set of non isolated equilibria and  $p(\tau)$ cannot be asymptotically stable. 
\end{proof}

\begin{theorem}\label{Prop15}
Assume that Hypotheses \ref{H1}, \ref{H2} and \ref{H4} are satisfied.  Then, for all $T>0$ the growth rate $\Lambda(m,T)$ of 
\eqref{eq3} satisfies 
\begin{equation}\label{m=infini}
\Lambda(\infty,T):=\lim_{m\to \infty}\Lambda(m,T)=\sum_{i=1}^n\overline{p_ir_i}.
\end{equation} 
\end{theorem}

\begin{proof}
The proof is given in Section \ref{ProofProp15}.
\end{proof}

\subsection{Double limits}
In \cite{BLSS,Katriel} the main tool to study the DIG phenomenon is the computation of the double limit
\begin{equation*}
\lim_{m\to 0}\lim_{T\to \infty}\Lambda(m,T)=\chi,\mbox{ where }\chi:=\int_0^1\max_{1\leq i\leq n}r_i(t)dt.
 \end{equation*}
Let us prove that this formula is also true in the more general context of this paper. According to Theorem \ref{Prop14}, when Hypotheses \ref{H1}, \ref{H2}, and \ref{H3} are satisfied  
\begin{equation}\label{mT=infini}
\Lambda(m,\infty):=\lim_{T\to \infty}\Lambda(m,T)=\overline{\lambda_{max}\left(R+mL\right)}.
\end{equation}
\begin{propo}\label{Prop14n}
Assume that Hypotheses \ref{H1}, \ref{H2},  and \ref{H3}  are satisfied. Then,
\begin{equation}\label{m=0T=infini}
\Lambda(0,\infty):=\lim_{m\to 0}\Lambda(m,\infty)=\chi:=\overline{\max_{1\leq i\leq n}r_i}.
\end{equation}
\end{propo}

\begin{proof}
The proof is the same as the proof of \cite[Eq. (17)]{BLSS}. Indeed, the proof in \cite{BLSS} for $\lim_{m\to0}\Lambda(m,\infty)$ only uses  the continuity of the spectral abscissa. For the details, we refer the reader to  \cite[Section 5.7]{BLSS}.
\end{proof}

According to Theorem \ref{Prop13}, when Hypotheses \ref{H1} and \ref{H2} are satisfied,
\begin{equation}\label{mT=0}
\Lambda(m,0):=\lim_{T\to 0}\Lambda(m,T)=\lambda_{max}\left(\overline{R+mL}\right).
\end{equation} 
We determine now the limits of 
$\Lambda(m,0)$ as $m$ tends to 0 or infinity.  

\begin{propo}\label{Prop13n}
Assume that Hypotheses \ref{H1} and \ref{H2}  are satisfied. Then,
\begin{equation}\label{m=0T=0}
\Lambda(0,0):=\lim_{m\to 0}\Lambda(m,0)=\max_i\overline{r}_i,
\end{equation}
and
\begin{equation}\label{m=infiniT=0}
\Lambda(\infty,0):=\lim_{m\to\infty}\Lambda(m,0)=\sum_{i=1}^nq_i\overline{r}_i,
\end{equation}
where $q\gg0$ in the Perron vector of $\overline{L}$, i.e. $q\in\Delta$
and $\overline{L}q=0$. Moreover, if $\overline{r}_i=\overline{r}$, for all $i$,  $\Lambda(m,0)=\overline{r}$ for all $m>0$, and, if the $\overline{r}_i$ are not equal, 
\begin{equation}\label{Lambda(m,0)convex}
\frac{d}{dm}\Lambda(m,0)< 0,\quad \frac{d^2}{dm^2}\Lambda(m,0)> 0.
\end{equation}
\end{propo}

\begin{proof}
For \eqref{m=0T=0} and \eqref{m=infiniT=0}, the proof is the same as the proof of \cite[Eq. (16)]{BLSS}. Indeed, the proof in \cite{BLSS} for $\lim_{m\to0}\Lambda(m,0)$ only uses  the continuity of the spectral abscissa, and the proof in \cite{BLSS} for $\lim_{m\to\infty}\Lambda(m,0)$ only use the fact that $\overline{L}$ is irreducible.  Moreover, for \eqref{Lambda(m,0)convex}, the proof is the same as the proof of \cite[Eq. (18)]{BLSS}. Indeed, the proof in \cite{BLSS} for the first and second derivatives of $\Lambda(m,0)$ only uses  the fact that $\overline{L}$ is irreducible. For the details, we refer the reader to  \cite[Section 5.7]{BLSS}. 
\end{proof}

\begin{figure}[ht]
\begin{center}
\includegraphics[width=10cm,
viewport=160 490 440 700]{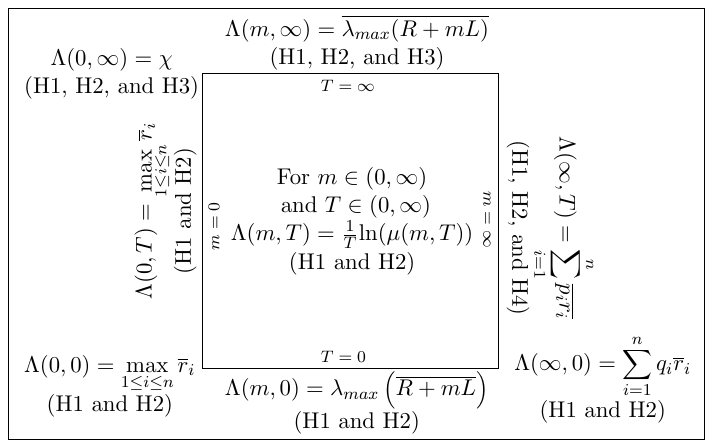}
\caption{The hypotheses under which the growth rate exists and its limits can be determined.   Compare with \cite[Fig. 1]{BLSS}, which was obtained under Hypothesis \ref{H1} and the hypothesis that the matrix $L(\tau)$ is irreducible for all $\tau\in[0,1]$.\label{figure1}}
\end{center}
\end{figure}

The results given by \eqref{m=0}, \eqref{m=infini}, \eqref{mT=infini}, 
\eqref{m=0T=infini}, \eqref{mT=0},  \eqref{m=0T=0}, and \eqref{m=infiniT=0} extend the results in \cite{BLSS} and are summarized in Figure \ref{figure1}. 

\begin{rem}\label{rem1}
Note that in the special case considered in \cite{BLSS}, 
\begin{equation}\label{m=infini,T=infini}
\Lambda(\infty,\infty):=\lim_{m\to \infty}\Lambda(m,\infty)=\sum_{i=1}^n\overline{p_ir_i},
\end{equation}
see \cite[Eq. (17)]{BLSS}.  On the other hand, if ${r}_i(\tau)={r}(\tau)$, for all $i$, 
$\Lambda(m,T)=\overline{r}$ for all $m>0$ and $T>0$, and, if the growth rates are not equal, 
\begin{equation}\label{Lambda(m,infini)convex}
\frac{d}{dm}\Lambda(m,\infty)< 0,\quad \frac{d^2}{dm^2}\Lambda(m,\infty)> 0,
\end{equation}
see \cite[Eq. (19)]{BLSS}.
However, the proofs of \eqref{m=infini,T=infini} and \eqref{Lambda(m,infini)convex} use the irreducibility of $L(\tau)$, and these results are not always satisfied in the case where only Hypothesis \ref{H2} is satisfied; see Remark \ref{NotStrictConvex}.
  \end{rem}

\subsection{Dispersal induced growth}
Following \cite{BLSS,Katriel}, we say that dispersal-induced growth (DIG) occurs if all patches are
sinks ($\overline{r}_i < 0$ for all $i$), but $\Lambda(m, T) > 0$ for some values of $m$ and $T$.

\begin{theorem}\label{Bornes2}
Assume that Hypotheses \ref{H1} and \ref{H2} are satisfied.
For all $m>0$ and $T>0$, 
$\sigma\leq \Lambda(m,T) \leq \chi,$ where
\begin{equation}
\label{chi_minmax1}
\sigma=\overline{\min_{1\leq i\leq n}r_i},
\quad
\chi=\overline{\max_{1\leq i\leq n}r_i}.
\end{equation}
\end{theorem}

\begin{proof}
Since the $i$th column of $A(t)=R(t)+mL(t)$ sums to $r_i(t)$, the numbers $\sigma$ and $\chi$ defined by \eqref{chi_minmax} are given by \eqref{chi_minmax1}
and the result follows from Corollary \ref{Bornes}.
\end{proof}

According to Theorem \ref{Bornes2}, a necessary condition for DIG to occur is $\chi>0$. In fact, this condition is also sufficient, as shown by the following result.

\begin{theorem}\label{DIGthmRed}
Assume that $\chi>0$. For all migration matrices satisfying Hypotheses \ref{H1}, \ref{H2} and \ref{H3}, there exist $m>0$ and $T>0$ such that $\Lambda(m,T)>0$.
{Therefore, if $\overline{r}_i<0$ for all $i$, DIG occurs if and only if $\chi>0$}.   
\end{theorem}

\begin{proof}
The result follows from 
$\sup_{m>0,T>0}\Lambda(m,T)=\chi,$
which is an consequence of 
\eqref{m=0T=infini} and Theorem \ref{Bornes2} 
\end{proof}

Theorem \ref{DIGthmRed} extends the result of \cite[Theorem 6]{BLSS} to the case where the irreducibility of the migration matrix  is replaced by the weaker Hypotheses \ref{H2} and \ref{H3} of the present paper.

\subsection{Dispersal induced  decay}
We say that dispersal-induced decay (DID) occurs if all patches are
sources ($\overline{r}_i > 0$ for all $i$), but $\Lambda(m, T) < 0$ for some values of $m$ and $T$. According to Theorem \ref{Bornes2}, a necessary condition for DID to occur is $\sigma<0$.
Theorem \ref{DIGthmRed} follows from the fact that $\chi$, the upper bound of $\Lambda(m,T)$, is in fact its supremum.  
There is no similar result when $\sigma<0$, because the lower bound $\sigma$ of $\Lambda(m,T)$ is far from being its infimum.

\begin{propo}
DID cannot occur if the migration matrix is time independent. 
\end{propo}
\begin{proof}
If the migration matrix $L$ is time independent, then, from \cite[Theorem 10]{BLSS}, 
$$\inf_{m>0,T>0}\Lambda(m,T)=\sum_{1}^nq_i\overline{r}_i,$$ 
where $q\gg 0$ is the Perron vector of $\overline{L}$. Therefore, if $\overline{r}_i\geq 0$ for all $i$, then $\sum_{1}^nq_i\overline{r}_i\geq 0$, so DID cannot occur.
\end{proof}

Our aim now is to describe a class of local growth functions $r_i(t)$ satisfying Hypothesis \ref{H1}, such that  there exist migration matrices for which 
$$\inf_{m>0,T>0}\Lambda(m,T)= \sigma.$$
Therefore, for this class of growth functions, if $\overline{r}_i>0$ for all $i$, and $\sigma<0$, DID occurs.
We assume that during the $k$-th season the minimum of the growth rates $r_i(t)$, $1\leq i\leq n$, is achieved on the same patch $i_k$. More precisely we make the following assumption.

\begin{hyp}\label{HypDID}
There exist a subdivision of $[0,1)$, $0=t_0<t_1<\ldots<t_N<t_{N+1}=1$, and integers 
$i_k\in\{1,\ldots,n\}$, $0\leq k\leq N$, such that for each $k$,   
\begin{equation}\label{HypDID1}
t\in[t_k,t_{k+1})\Longrightarrow
\min_{1\leq i\leq n}r_i(t)=r_{i_k}(t).
\end{equation}
\end{hyp}

\begin{lem}
Hypothesis \ref{HypDID} is satisfied if the growth functions are piecewise constant. 
\end{lem}

\begin{proof}Without loss of generality we can assume that the functions $r_i$, $1\leq i\leq n$,  are piecewise constant on the same subdivision 
$0<t_1<\ldots<t_p<t_{N+1}=1,$
(if not, it suffices to take a subdivision which is finer than the subdivisions for the functions $r_i$). Let 
$t\in[t_k,t_{k+1})$, and $i_k$ be such that
\begin{equation}\label{tk}
\min_{1\leq i\leq n} r_j(t_k) = r_{i_k}(t_k).
\end{equation}
Since the functions $r_i$ are constant on $[t_k,t_{k+1})$, the property \eqref{tk} remains true for all $t\in[t_k,t_{k+1})$, which proves \eqref{HypDID1}.
\end{proof}

Note that Hypothesis \ref{HypDID} is not always  satisfied when the $r_i$ functions verify only Hypothesis \ref{H1}. For example, the growth function $r_1(t)$ defined on $[0,1)$ by $r_1(0)=0$ and 
$r_1(t)=t\sin\left(\frac{2\pi}{t}\right)$ if $t>0$, and the growth function $r_2(t)=0$, do not satisfy \eqref{HypDID1}.

Let $I=\{i\in\mathbb{N}:1\leq i\leq n\}$.
We denote by $J\subset I$ the set of indices that achieve the minimum of the $r_i(t)$, i.e. if $j\in J$ then there exists an interval $[t_k,t_{k+1})$ of the subdivision such that $j=i_k$:
$J=\{i_k\in I:0\leq k\leq N\}.$
Let $K=I\setminus J$. The subset $K\subset I$ is the set of indices that never achieve the minimum. Note that $J$ is not empty, while $K$ can be empty or not.
Assume that Hypothesis \ref{HypDID} is satisfied. Let $L(t)=(\ell_{ij}(t))$ be the piecewise constant 1-periodic migration matrix, whose off-diagonal elements $\ell_{ij}(t)$ are defined on the subdivision $0=t_0<t_1<\ldots<t_N<t_{N+1}=1$ of $[0,1)$ as follows.
\begin{itemize}
\item
If $K=\emptyset$, then, for $0\leq k\leq N$ and $t\in[t_k,t_{k+1})$, $\ell_{ij}(t)$ are given by
\begin{equation}\label{L(t)Kempty}
\ell_{ij}(t)=
\left\{
\begin{array}{l}
1\mbox{ if } i=i_k\mbox{ and }j\neq i,\\
0\mbox{ if } i\neq i_k\mbox{ and }j\neq i.
\end{array}
\right.
\end{equation}
where $i_k$, for $0\leq k\leq N$, are as in Hypothesis \ref{HypDID}. Therefore, for $t\in[t_k,t_{k+1})$, $0\leq k\leq N$, migration is only to the site $i_k$, that achieve the minimum of the growth rates $r_i(t)$ on $[t_k,t_{k+1})$. 
\item
If $K$ is not empty, over the interval $[t_0,t_{1})$
we add small migration terms to the patches $i\in K$. More precisely, for $t\in[t_0,t_{1})$, $\ell_{ij}(t)$ are given by 
\begin{equation}\label{L(t)Knotempty}
\ell_{ij}(t)=
\left\{
\begin{array}{l}
1\mbox{ if } i=i_0\mbox{ and }j\neq i,\\
\varepsilon\mbox{ if } i\in K \mbox{ and }j\neq i,\\
0\mbox{ if } i\neq i_0, i\notin K \mbox{ and }j\neq i.
\end{array}
\right.
\end{equation}
For $k\geq 1$,  $\ell_{ij}(t)$ are defined on  $[t_k,t_{k+1})$ by \eqref{L(t)Kempty}. Therefore, 
for $t\in[t_0,t_{1})$, migration is with rate $1$ to the site $i_0$, that achieve the minimum of the growth rates $r_i(t)$ on $[t_0,t_{1})$, and also, with rate $\varepsilon$, to the sites $i\in K$,  where $\varepsilon$ is small enough.
For $t\in[t_k,t_{k+1})$, $1\leq k\leq N$, migration is only to the site $i_k$, that achieve the minimum of the growth rates $r_i(t)$ on $[t_k,t_{k+1})$.
\end{itemize}

\begin{propo}\label{PropL(t)}
The matrix $L$, defined above, satisfies Hypotheses \ref{H1}, \ref{H2} and \ref{H4}.
\end{propo}

\begin{proof}
The matrix $\ell(t)$ is piecewise constant. Therefore, Hypothesis \ref{H1} is satisfied.
 
Over $[t_k,t_{k+1})$, the off-diagonal elements of the $i_k$th-row of matrix $L(t)$ are equal to 1. Hence, if $K=\emptyset$, or equivalently $J=I$, 
the off-diagonal elements of any row of the matrix $L(t)$ are positive on some subinterval  of the subdivision. Therefore, the averaged matrix $\overline{L}$ has positive off-diagonal elements, and is therefore irreducible.
In the case where $K\neq\emptyset$, on $[t_0,t_1)$, the off-diagonal elements of the $i_0$th-row of matrix $L(t)$ are equal to 1 and the off-diagonal elements of the $i$th-row, $i\in K$, of matrix $L(t)$ are equal to $\varepsilon$. Moreover, for any $j\in J$, the off-diagonal elements of the $j$th-row, of matrix $L(t)$ are equal to 1 on some subinterval of the subdivision.
Hence, the off-diagonal elements of any row of the matrix $L(t)$ are positive on some subinterval of the subdivision. Therefore, the averaged matrix $\overline{L}$ has positive off-diagonal elements, and is therefore irreducible. This proves that Hypothesis \ref{H2} is satisfied.

In the case where $K=\emptyset$,  the migration matrix is defined by \eqref{L(t)Kempty} over $[t_k,t_{k+1})$. The spectrum of $L(t)$ is $\sigma(L(t))=\{0,-1\}$, where 0 is a simple eigenvalue and $-1$ is of multiplicity $n-1$. Therefore, Hypothesis \ref{H4} is satisfied.
Similarly, in the case where $K\neq\emptyset$, over $[0,t_1)$, the migration matrix is defined by
\eqref{L(t)Knotempty}. The spectrum of $L(t)$ is $\sigma(L(t))=\{0,-1-|K|\varepsilon\}$, 
where $|K|$ is the number of elements of $K$,
0 is a simple eigenvalue and $-1-|K|\varepsilon$ is of multiplicity $n-1$. Therefore, Hypothesis \ref{H4} is satisfied.
\end{proof}

\begin{theorem}\label{DIDthmRed}
Assume that the growth rates $r_i(t)$ satisfy Hypotheses \ref{H1} and \ref{HypDID}, and $\sigma<0$. There exist migration matrices satisfying Hypotheses \ref{H1} and \ref{H2}, and there exist $m>0$ and $T>0$ such that $\Lambda(m,T)<0$.
\end{theorem}

\begin{proof}
Let $L(t)=(\ell_{ij}(t))$ be the migration matrix defined above. Using Proposition \ref{PropL(t)}, the migration matrix $L(t)$ satisfies Hypotheses \ref{H1} and \ref{H2}, and hence, $\Lambda(m,T)$ exists. Since the migration matrix also satisfies Hypothesis \ref{H4} and   
 using Theorem \ref{Prop15},  we obtain
$$\lim_{m\to\infty}\Lambda(m,T)=\sum_{i=1}^n\overline{p_ir_i}=\sum_{i=1}^n\int_0^1p_i(t)r_i(t)dt=\sum_{i=1}^n\sum_{k=0}^N\int_{t_k}^{t_{k+1}}p_i(t)r_i(t)dt,$$ 
where $p(t)$ is the Perron-Frobenius vector of $L(t)$.  
If $K=\emptyset$, since $L(t)$ is defined by
\eqref{L(t)Kempty}, for $0\leq k\leq N$
its Perron-Frobenius vector is given over $[t_k,t_{k+1})$  by
\begin{equation}\label{p(t)}
p_{i}(t)=
\left\{
\begin{array}{l}
1\mbox{ if } i=i_k,\\
0\mbox{ if } i\neq i_k.
\end{array}
\right.
\end{equation}
Therefore,
\begin{align*}
\sum_{i=1}^n\sum_{k=0}^N\int_{t_k}^{t_{k+1}}p_i(t)r_i(t)dt
&=
\sum_{k=0}^N\int_{t_k}^{t_{k+1}}\sum_{i=1}^np_i(t)r_i(t)dt
=
\sum_{k=0}^N\int_{t_k}^{t_{k+1}}r_{i_k}(t)dt\\
&=\sum_{k=0}^N\int_{t_k}^{t_{k+1}}\min_ir_{i}(t)dt
=\int_0^1\min_ir_{i}(t)dt=
\overline{\min_{1\leq j\leq n}r_j}=\sigma
\end{align*}
Thus, $\Lambda(\infty,T)=\sigma$. Therefore, if $\sigma<0$, there exist $m>0$ and $T>0$ such that $\Lambda(m,T)<0$.

If $K\neq\emptyset$, 
since $L(t)$ is defined by \eqref{L(t)Knotempty}, its 
Perron-Frobenius vector $p(t)$ is defined over $[t_0,t_1)$ by
$$ 
p_{i}(t)=
\left\{
\begin{array}{lcl}
1/(1+\varepsilon|K|)&\mbox{if}& i=i_0,\\
\varepsilon/(1+\varepsilon|K|)&\mbox{if}& i\in K,\\
0&\mbox{if}& i\neq i_k\mbox{ and } i\notin K,
\end{array}
\right.
$$
where $|K|$ is the number of elements of $K$.
For $1\leq k\leq N$, $p(t)$ is given by \eqref{p(t)} over $[t_k,t_{k+1})$. 
A simple calculation, similar to the one above for  $K=\emptyset$,  shows that
$$\sum_{i=1}^n\sum_{k=0}^N\int_{t_k}^{t_{k+1}}p_i(t)r_i(t)dt=\sigma+\frac{\varepsilon}{1+\varepsilon|K|}\int_{t_0}^{t_1}\left(\sum_{i\in K}r_i(t)-|K|r_{i_0}(t)\right)dt.$$
Thus, $\Lambda(\infty,T)=\sigma+O(\varepsilon)$. Therefore, if $\sigma<0$ and $\varepsilon$ is small enough, there exist $m>0$ and $T>0$ such that $\Lambda(m,T)<0$. 
\end{proof}

\section{Examples}\label{Numerical}
\label{INNex}

We consider in this section examples with constant growth rate and migration during each season. Since the system is piecewise constant, we will be able to use the Maple software to compute
the monodromy matrix, and
then determine its Perron root. When the growth rates during each season are fixed, we can plot the graph of the function $\Lambda(m,T)$ 
showing the behavior of the growth rate as a function of $m$ and
$T$ for the whole range of values of these parameters. We will compare this plot with our asymptotic formulas  for the limits of $\Lambda(m,T)$, when $m$ or $T$ tends to 0 or infinity. 

\subsection{Two patches case}\label{TPC}

We consider examples with two patches and two seasons, where the growth rate is constant and the migration is constant and unidirectional during each 
season.  

\subsubsection{Unidirectional migration to the most unfavourable patch}\label{MUP}

Consider the $T$-periodic linear system given by
\begin{equation}\label{eq35}
\frac{dx}{dt}=\left\{
\begin{array}{l}
A_1x,\mbox{ if }t\in[0,T/2),\\  
A_2x,\mbox{ if }t\in[T/2,T),
\end{array}
\right.
\mbox{ with }
A_1=\left[
\begin{array}{cc}
a_1-m&0\\
m&b_1
\end{array}
\right],
~
A_2=\left[
\begin{array}{cc}
b_2&m\\
0&a_2-m
\end{array}
\right].
\end{equation} 
The migration, which is from the patch where the growth rate is equal to $a_i$ to the patch where it is equal to $b_i\leq a_i$,  is unidirectional toward the unfavourable patch, see Fig. \ref{figure2}(left). 
We will consider in Section \ref{MFP} the case where the migration is unidirectional toward the favourable patch, corresponding to Fig. \ref{figure2}(right). 
Note that in both cases
$\overline{r}_1=\frac{a_1+b_2}{2}$ and 
$\overline{r}_2=\frac{b_1+a_2}{2}$. 

\begin{figure}[ht]
\begin{center}
\includegraphics[width=10cm,
viewport=160 570 440 675]{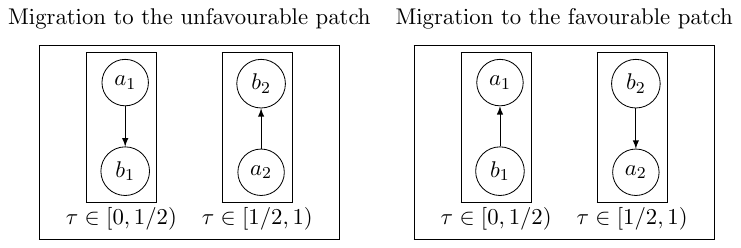}
\caption{The case of two patches with growth rates $a_i\geq b_i$, $i=1,2$, and two seasons. 
\label{figure2}}
\end{center}
\end{figure}

\begin{propo}\label{Prop16}
The system \eqref{eq35} admits a growth rate $\Lambda(m,T)$ and
$$
\begin{array}{l}
\Lambda(0,T)=\max(\overline{r}_1,\overline{r}_2),
\quad
\Lambda(\infty,T)=\frac{b_1+b_2}{2},
\\[1mm]
\Lambda(m,0)=\frac{1}{2}
\left(\overline{r}_1+\overline{r}_2-m+\sqrt{\left(\overline{r}_1-\overline{r}_2\right)^2+m^2}\right),\\
\Lambda(0,0)=\max(\overline{r}_1,\overline{r}_2),
\quad
\Lambda(\infty,0)=\frac{\overline{r}_1+\overline{r}_2}{2}.
\end{array}
$$
Moreover, 
if $a_1\geq b_1$ and $a_2\geq b_2$, then 
$\sigma=\frac{b_1+b_2}{2}$,  $\chi=\frac{a_1+a_2}{2}$, and 
$$
\Lambda(m,\infty)=
\left\{\begin{array}{ll}
\chi-{m}&\mbox{if }0<m<a_1-b_1,\\
\frac{a_2+b_1-m}{2}&\mbox{if }a_1-b_1\leq m
\leq a_2-b_2,\\
\sigma&\mbox{if } m
> a_2-b_2,
\end{array}
\right.
$$
where we assumed, without loss of generality, that $a_1-b_1\leq a_2-b_2$, i.e. $\overline{r}_1\leq \overline{r}_2$.
\end{propo}

\begin{proof}
The migration matrix is
\begin{equation*}
L(\tau)=\left[\begin{array}{cc}
-1&0\\
1&0
\end{array}\right] 
\mbox{ for }\tau\in[0,1/2)\quad\mbox{ and }\quad  
L(\tau)=\left[\begin{array}{cc}
0&1\\
0&-1
\end{array}\right]
\mbox{ for }\tau\in[1/2,1).
\end{equation*}
Since $\overline{L}$ is irreducible,  from Theorem \ref{thm2} we deduce that $\Lambda(m,T)$ exists. From Proposition \ref{Prop12m=0} we deduce that $\Lambda(0,T)=\max(\overline{r}_1,\overline{r}_2)$. 
 
The eigenvalues of $L(\tau)$ are 0 and $-1$. Thus, Hypothesis \ref{H4} is satisfied. Therefore, 
using Theorem \ref{Prop15}, and the fact that the Perron-Frobenius vector 
$p(\tau)=\left(p_1(\tau),p_2(\tau)\right)$ 
 of $L(\tau)$ is given by
$$
p_1(\tau)=\left\{
\begin{array}{l}
0\mbox{ if }\tau\in[0,1/2),\\
1\mbox{ if }\tau\in[1/2,1),
\end{array}
\right.
\quad
p_2(\tau)=\left\{
\begin{array}{l}
1\mbox{ if }\tau\in[0,1/2),\\
0\mbox{ if }\tau\in[1/2,1),
\end{array}
\right.
$$
we obtain
$\Lambda(\infty,T)=\overline{p_1r_1+p_2r_2}=\frac{b_1+b_2}{2}$.

 Now we prove the formulas for $\Lambda(m,0)$ and its limits when $m\to 0$ and $m\to \infty$.
The average of $A(\tau)$ is
$$\overline{A}=
\left[
\begin{array}{cc}
\overline{r}_1-m/2&m/2\\
m/2&\overline{r}_2-m/2
\end{array}
\right]
$$
Using 
Theorem \ref{Prop13}, 
$$\Lambda(m,0)=\lambda_{max}(\overline{A})=
\frac{1}{2}
\left(\overline{r}_1+\overline{r}_2-m+\sqrt{\left(\overline{r}_1-\overline{r}_2\right)^2+m^2}\right).
$$

The Perron-Frobenius vector $q$ of $\overline{L}$ is given by $q_1=q_2=1/2$. Therefore, using Proposition \ref{Prop13n} we obtain  
$$
\Lambda(0,0)=
\max\left(\overline{r}_1,\overline{r}_2\right),
\quad
\Lambda(\infty,0)=\frac{\overline{r}_1+\overline{r}_2}{2}.$$

Now we prove the formulas for $\Lambda(m,\infty)$. The eigenvalues of $A_1$ are $a_1-m$ and $b_1$ and those of $A_2$ are $a_2-m$ and $b_2$. Assume that  $a_1\geq b_1$ 
and $a_2\geq b_2$. Then,
$$\lambda_{max}(A_1)=
\left\{\begin{array}{lcl}
a_1-m&\mbox{if}&0<m<a_1-b_1,\\
b_1&\mbox{if}&m\geq a_1-b_1,
\end{array}
\right.
~
\lambda_{max}(A_2)=
\left\{\begin{array}{lcl}
a_2-m&\mbox{if}&0<m<a_2-b_2,\\
b_2&\mbox{if}&m\geq a_2-b_2.
\end{array}
\right.
$$
Without loss of generality, we assume that $a_1-b_1\leq a_2-b_2$.

If $0<m<a_1-b_1$, then $\lambda_{max}(A_1)=a_1-m$, $\lambda_{max}(A_2)=a_2-m$, and
their corresponding Perron-Frobenius vectors are given by
\begin{equation}\label{v1v2}
v_1=\left(
\frac{a_1-b_1-m}{a_1-b_1},\frac{m}{a_1-b_1}\right),
\quad
v_2=\left(
\frac{m}{a_2-b_2},\frac{a_2-b_2m}{a_2-b_2}\right),
\end{equation}
respectively.
Let us prove that Hypothesis \ref{H3} is satisfied. Indeed the differential system on the simplex (parametrized by $\theta_2\in[0,1]$), corresponding to the matrix $A_1$, is
\begin{equation}\label{eqSimp3}
\frac{d\theta_2}{dt}=(1-\theta_2)(m-(a_1-b_1)\theta_2).
\end{equation}
Since $0<m<a_1- b_1$, it admits $\theta_2=1$ and $\theta_2=\frac{m}{a_1-b_1}\in(0,1)$ as equilibria, the first being unstable and the second being globally asymptotically stable in the interior of the simplex. Therefore, $v_1$ is GAS in the interior of the simplex.  Similarly, the
differential system on the simplex (parametrized by $\theta_1\in[0,1]$), corresponding to the matrices $A_2$, is
\begin{equation}\label{eqSimp4}
\frac{d\theta_1}{dt}=(1-\theta_1)(m-(a_2-b_2)\theta_1).
\end{equation}
Since $0<m<a_2- b_2$, it admits $\theta_1=1$ and $\theta_1=\frac{m}{a_2-b_2}\in(0,1)$ as equilibria, the first being unstable and the second being globally asymptotically stable in the interior of the simplex. Therefore, $v_2$ is GAS in the interior of the simplex. 
 Using Theorem \ref{Prop14}, 
$$\Lambda(m,\infty)=
\frac{\lambda_{max}(A_1)+\lambda_{max}(A_2)}{2}=
\frac{a_1-m+a_2-m}{2}=
\chi-m.$$

If $a_1-b_1< m<a_2-b_2$, then $\lambda_{max}(A_1)=b_1$, $\lambda_{max}(A_2)=a_2-m$, and
their corresponding Perron-Frobenius vectors are 
$v_1=\left(0,1\right)$ and $v_2$, given by \eqref{v1v2},
respectively. We have already seen that $v_2$ is GAS for the differential equation \eqref{eqSimp4}. 
Since $m>a_1- b_1$, the differential equation \eqref{eqSimp3} admits only $\theta_2=1$ as a GAS equilibrium. Therefore, $v_1=(0,1)$ is GAS in the interior of the simplex.  
 Using Theorem \ref{Prop14}, 
$$\Lambda(m,\infty)=
\frac{\lambda_{max}(A_1)+\lambda_{max}(A_2)}{2}=
\frac{b_1+a_2-m}{2}.$$

Finally, if $m>a_2-b_2$, then $\lambda_{max}(A_1)=b_1$, $\lambda_{max}(A_2)=b_2$, and
their corresponding Perron-Frobenius vectors are given by
$v_1=\left(0,1\right)$ and $v_2=(1,0)$,
respectively. We have already seen that $v_1$ is GAS for the differential equation \eqref{eqSimp3}. 
Since $m>a_2- b_2$, the differential equation \eqref{eqSimp4} admits only $\theta_1=1$ as a GAS equilibrium. Therefore, $v_2=(1,0)$ is GAS in the interior of the simplex.  
 Using Theorem \ref{Prop14}, 
$$\Lambda(m,\infty)=
\frac{\lambda_{max}(A_1)+\lambda_{max}(A_2)}{2}=
\frac{b_2+b_1}{2}=\sigma.$$
This ends the proof for the formulas giving $\Lambda(m,\infty)$.
\end{proof}

\begin{figure}[ht]
\begin{center}
\includegraphics[width=9cm,
viewport=170 355 430 670]{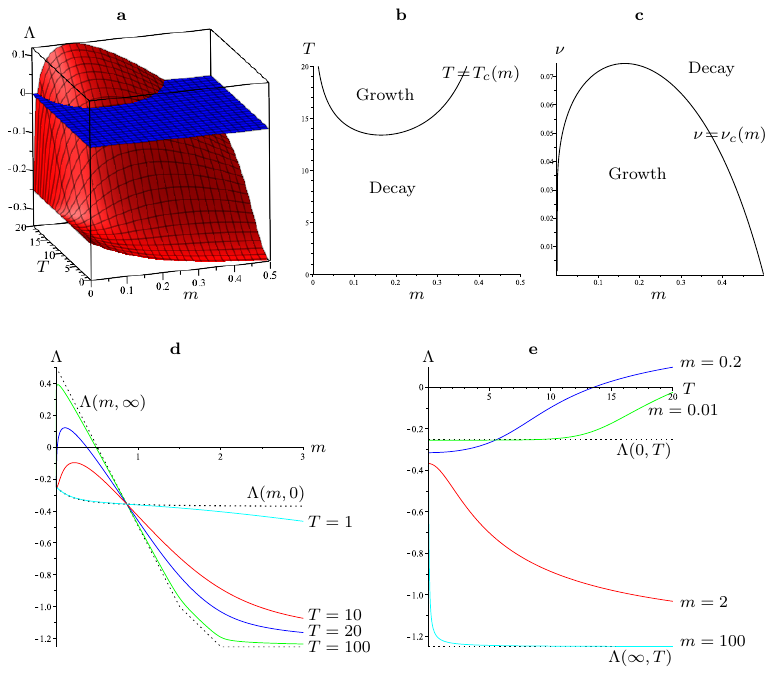}
\caption{
The system is
\eqref{eq35}, with
$a_1=1/2$, $b_1=-1$, $a_2= 1/2$ and $b_2=-3/2$.
{\bf a} The graph of $(m,T)\mapsto\Lambda(m,T)$. 
{\bf b} The set $\Lambda(m,T)=0$.
{\bf c} The set $\Lambda(m,1/\nu)=0$ in the $(m,\nu)$ parameter-plane.
{\bf d} Graphs of $m\mapsto\Lambda(m,T)$ with the indicated values of $T$.
{\bf e} Graphs of $T\mapsto\Lambda(m,T)$ with the indicated values of $m$.   The limits $\Lambda(0,T)$, $\Lambda(m,0)$, $\Lambda(\infty,T)$ and $\Lambda(m,\infty)$ are  given in Table \ref{TableFigure1}.
 \label{figure3}}
\end{center}
\end{figure}

\begin{table}
\caption{Limits of $\Lambda(m,T)$ for the system
considered in Figure \ref{figure3}. Here, $\overline{r}_1= -1/2$, $\overline{r}_2=-1/4$, $\chi=1/2$ and $\sigma=-5/4$.}\label{TableFigure1}
\begin{tabular}{l}
\toprule
$\Lambda(0,T)=-1/4$, 
\quad
$\Lambda(\infty,T)=-5/4$,
\qquad
$\Lambda(0,0)=-1/4$
\quad 
$\Lambda(\infty,0)=-3/8$,
\\[2mm]
$\Lambda(m,0)=-\frac{3}{8}-\frac{m}{2}+\frac{1}{8}\sqrt{1+16m^2}$,
\quad
$
\Lambda(m,\infty)=
\left\{\begin{array}{ll}
1/2-{m}&\mbox{if }0<m<3/2,\\
-1/4-m/2&\mbox{if }3/2\leq m
\leq 2,\\
-5/4&\mbox{if } m
> 2,
\end{array}
\right.
$\\[1mm]
$\Lambda(m^*,\infty)=0$ for $m^*=1/2$.\\
\botrule
\end{tabular}
\end{table}

For the system considered in  Proposition \ref{Prop16},  $\Lambda(0,\infty)=\chi$. Therefore, using Theorem \ref{DIGthmRed}, if $\overline{r}_1<0$, $\overline{r}_2<0$ and $\chi>0$, then the patches are sinks and DIG occurs. This result rigorously establish the conclusions made numerically in  \cite[Section 4.5.1]{BLSS}. 
Let us illustrate the results with an example where DIG occurs.

Consider the system given in Fig. \ref{figure3}. Using the formulas in Proposition \ref{Prop16}, we obtain the limits given in Table \ref{TableFigure1}.  We show in Fig. 
\ref{figure3}(a) the graph of the growth rate $\Lambda(m,T)$ and in Fig. 
\ref{figure3}(b) the set where $\Lambda(m,T)=0$. The figure shows that there exists a critical curve $T=T_c(m)$, defined for $0<m<m^*$, such that $T_c(0)=T_c(m^*)=\infty$ and growth occurs if and only if $T>T_c(m)$. This behavior is best seen in Fig. \ref{figure3}(c) where we display the set where growth occurs in the $(m,\nu)$-plane,
where $\nu= 1/T$ is the frequency. The figure shows that $m^*=1/2$, which is the value of $m$ for which $\Lambda(m,\infty)=0$. Note that the critical curve $v=v_c(m)$ has asymptotic behavior, of the form $m\sim e^{-k/\nu}$, that is $m$ becomes exponentially small in $1/\nu$ near the origin, a fact that was already numerically observed by Katriel \cite{Katriel} and Benaim et al. \cite{BLSS}. It was proved in a particular case, see \cite[Proposition 2.9)]{benaim}.  

The functions $\Lambda(m,0)$ and $\Lambda(m,\infty)$ are depicted in dotted line on Fig. 
\ref{figure3}(d) and the functions $\Lambda(0,T)$ and $\Lambda(\infty,T)$ are depicted in dotted line on Fig. 
\ref{figure3}(e). These figures show the graphs of functions $m\mapsto\Lambda(m,T)$ for various values of $T$, and the graphs of functions $T\mapsto\Lambda(m,T)$ for various values of $m$ and illustrate the convergence toward $\Lambda(m,0)$ and $\Lambda(m,\infty)$ as $T$ tends to 0 and $\infty$ and the convergence toward $\Lambda(0,T)$ and $\Lambda(\infty,T)$ as $m$ tends to 0 and $\infty$.

\subsubsection{Unidirectional migration to the most favourable patch}\label{MFP}

Consider the $T$-periodic linear system given by
\begin{equation}\label{eq35BIS}
\frac{dx}{dt}=\left\{
\begin{array}{l}
A_1x,\mbox{ if }t\in[0,T/2),\\  
A_2x,\mbox{ if }t\in[T/2,T),
\end{array}
\right.
\mbox{ with }
A_1=\left[
\begin{array}{cc}
a_1&m\\
0&b_1-m
\end{array}
\right],
~
A_2=\left[
\begin{array}{cc}
b_2-m&0\\
m&a_2
\end{array}
\right].
\end{equation}
Now the  migration is unidirectional from the patch where the local growth rate is $b_i$ to the patch where it is $a_i\geq b_i$, see Fig. \ref{figure2}(right). 
We have the following result whose proof is similar to the proof of Proposition \ref{Prop16}.

\begin{propo}\label{Prop16n}
The system \eqref{eq35BIS} admits a growth rate $\Lambda(m,T)$ and we have 
the same formulas for $\Lambda(0,T)$ and $\Lambda(m,0)$ as in Proposition \ref{Prop16}, while the limit $\Lambda(\infty,T)$ is now given by
$
\Lambda(\infty,T)=\frac{a_1+a_2}{2}.
$
Moreover, 
if $a_1\geq b_1$ and $a_2\geq b_2$, then 
$\sigma=\frac{b_1+b_2}{2}$, $\chi=\frac{a_1+a_2}{2}$, and 
$\Lambda(m,\infty)=\chi.$ 
\end{propo}

\begin{rem}
The limit $\Lambda(0,\infty)=\chi$ can be obtained directly from the formulas in Propositions \ref{Prop16} and \ref{Prop16n}, giving 
$\Lambda(m,\infty)$
or by using Proposition \ref{Prop14n}. 
\end{rem}

\begin{rem}\label{NotStrictConvex}
Note that from the formulas for 
$\Lambda(m,\infty)$ in Propositions \ref{Prop16} and \ref{Prop16n}, we deduce that 
 $\Lambda(\infty,\infty)=\overline{p_1r_1+p_2r_2}=\Lambda(\infty,T)$, a formula which is always true in the case where $L(\tau)$ is irreducible for all $\tau$, see Remark \ref{rem1}. However the strict convexity stated in \eqref{Lambda(m,infini)convex} {does not hold}, since $\Lambda(m,\infty)$ is piecewise linear or constant.
\end{rem}


\begin{figure}[ht]
\begin{center}
\includegraphics[width=9cm,
viewport=170 355 430 670]{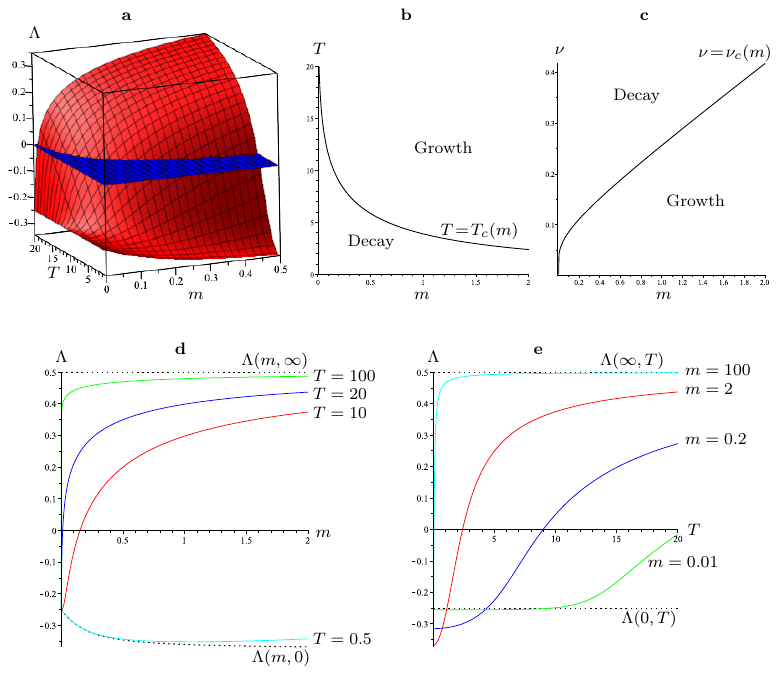}
\caption{The system is
\eqref{eq35BIS}, with
$a_1=1/2$, $b_1=-1$, $a_2= 1/2$ and $b_2=-3/2$. {\bf a} The graph of $(m,T)\mapsto\Lambda(m,T)$. 
{\bf b} The set $\Lambda(m,T)=0$.
{\bf c} The set $\Lambda(m,1/\nu)=0$ in the $(m,\nu)$ parameter-plane. 
{\bf d} Graphs of $m\mapsto\Lambda(m,T)$ with the indicated values of $T$.
{\bf e} Graphs of $T\mapsto\Lambda(m,T)$ with the indicated values of $m$. The limits $\Lambda(0,T)$ and  
$\Lambda(m,0)$ are as in Table \ref{TableFigure1}, and $\Lambda(\infty,T)=\Lambda(m,\infty)=1/2$. 
\label{figure4}}
\end{center}
\end{figure}

For the systems considered in  Proposition \ref{Prop16n},  $\Lambda(0,\infty)=\chi$. Therefore, using Theorem \ref{DIGthmRed}, if $\overline{r}_1<0$, $\overline{r}_2<0$ and $\chi>0$, then the patches are sinks and DIG occurs. This result rigorously establish the conclusions made numerically in  \cite[Section 4.5.1]{BLSS}. 

Let us illustrate the result with an example where DIG occurs.
We consider the system given in Fig. \ref{figure4}. Using the formulas in Proposition \ref{Prop16n}, we obtain the same expressions for the limits $\Lambda(0,T)$ and $\Lambda(m,0)$ as those given in Table \ref{TableFigure1}, while the limits $\Lambda(\infty,T)$ and $\Lambda(m,\infty)$ are now given by $\Lambda(\infty,T)=\Lambda(m,\infty)=\chi$.  
We show in Fig. 
\ref{figure3}(d) the graph of the growth rate $\Lambda(m,T)$ and in Fig. 
\ref{figure3}(e) the set where $\Lambda(m,T)=0$. The figure shows that there exists a critical curve $T=T_c(m)$, defined for $m>0$, such that $T_c(0)=\infty$, $T_c(\infty)=0$,  and growth occurs if and only if $T>T_c(m)$. This behavior is best seen in Fig. \ref{figure4}(c) where we display the set where growth occurs in the $(m,\nu)$-plane,
where $\nu= 1/T$ is the frequency. The figure shows that the critical curve $v=v_c(m)$ verifies $\nu_c(0)=0$, $\nu_c(\infty)=0$ and $m$ becomes exponentially small in $1/\nu$ near the origin. Determining the asymptotic behavior in $m=0$ and $m=\infty$ of the critical curve is an open problem that deserves more attention.  

The functions $\Lambda(m,0)$ and $\Lambda(m,\infty)$ are depicted in dotted line on Fig. 
\ref{figure4}(d) and the functions $\Lambda(0,T)$ and $\Lambda(\infty,T)$ are depicted in dotted line on Fig. 
\ref{figure4}(e). These figures show the graphs of functions $m\mapsto\Lambda(m,T)$ for various values of $T$, and the graphs of functions $T\mapsto\Lambda(m,T)$ for various values of $m$ and illustrate the convergence toward $\Lambda(m,0)$ and $\Lambda(m,\infty)$ as $T$ tends to 0 and $\infty$ and the convergence toward $\Lambda(0,T)$ and $\Lambda(\infty,T)$ as $m$ tends to 0 and $\infty$.

\subsubsection{Dispersal induced decay}

The system considered in Proposition \ref{Prop16} provides examples for which DID can occur. Indeed since $\Lambda(\infty,T)=\sigma$, the lower bound $\sigma$ of $\Lambda(m,T)$ is its infimum. Hence, if $\overline{r}_1>0$, $\overline{r}_2>0$ and $\sigma<0$, then the patches are sources and DID occurs.

\begin{figure}[ht]
\begin{center}
\includegraphics[width=9cm,
viewport=170 355 430 670]{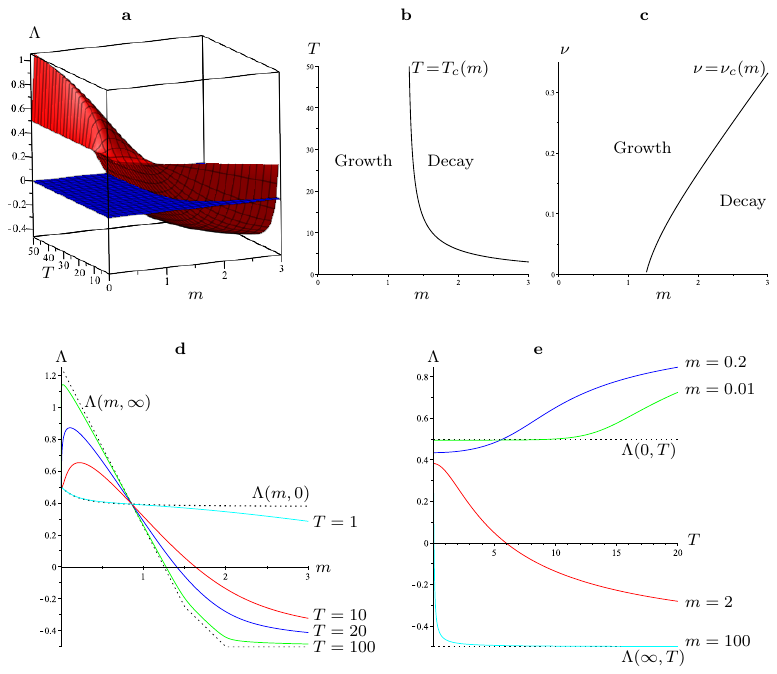}
\caption{
The system is \eqref{eq35}, with
$a_1=1$, $b_1=-1/2$, $a_2= 3/2$ and $b_2=-1/2$.
{\bf a} The graph of $(m,T)\mapsto\Lambda(m,T)$. 
{\bf b} The set $\Lambda(m,T)=0$.
{\bf c} The set $\Lambda(m,1/\nu)=0$ in the $(m,\nu)$ parameter-plane.
{\bf d} Graphs of $m\mapsto\Lambda(m,T)$ with the indicated values of $T$.
{\bf e} Graphs of $T\mapsto\Lambda(m,T)$ with the indicated values of $m$. The limits $\Lambda(0,T)$, $\Lambda(m,0)$, $\Lambda(\infty,T)$ and $\Lambda(m,\infty)$ are  given in Table \ref{TableFigure1DID}
\label{figure5}}
\end{center}
\end{figure}

\begin{table}
\caption{Limits of $\Lambda(m,T)$ for the system
considered in Fig. \ref{figure5}. Here, $\overline{r}_1= 1/4$, $\overline{r}_2=1/2$, $\chi=5/4$ and $\sigma=-1/2$.}\label{TableFigure1DID}
\begin{tabular}{l}
\toprule
$\Lambda(0,T)=1/2$, 
\quad
$\Lambda(\infty,T)=-1/2$,
\qquad
$\Lambda(0,0)=1/2$
\quad 
$\Lambda(\infty,0)=3/8$,
\\[2mm]
$\Lambda(m,0)=\frac{3}{8}-\frac{m}{2}+\frac{1}{8}\sqrt{1+16m^2}$,
\quad
$
\Lambda(m,\infty)=
\left\{\begin{array}{ll}
5/4-{m}&\mbox{if }0<m<3/2,\\
1/2-m/2&\mbox{if }3/2\leq m
\leq 2,\\
-1/2&\mbox{if } m
> 2,
\end{array}
\right.
$\\[1mm]
$\Lambda(m^*,\infty)=0$ for $m^*=5/4$.\\
\botrule
\end{tabular}
\end{table}

Let us illustrate the result with an example where DID occurs. We consider the system given in Fig. \ref{figure5}. Using the formulas in Proposition \ref{Prop16}, we obtain the limits given in Table \ref{TableFigure1DID}.  We show in Fig. 
\ref{figure5}(a) the graph of the growth rate $\Lambda(m,T)$ and in Fig. 
\ref{figure5}(b) the set where $\Lambda(m,T)=0$. The figure shows that there exists a critical curve $T=T_c(m)$, defined for $m>m^*$, such that $T_c(m^*)=\infty$, $T_c(\infty)=0$, and decay occurs if and only if $T>T_c(m)$. This behavior is best seen in Fig. \ref{figure5}(c) where we display the set where growth occurs in the $(m,\nu)$-plane,
where $\nu= 1/T$ is the frequency. The figure shows that $m^*=5/4$, which is the value of $m$ for which $\Lambda(m,\infty)=0$. Determining the asymptotic behavior in $m=m^*$ and $m=\infty$ of the critical curve is an open problem.

The functions $\Lambda(m,0)$ and $\Lambda(m,\infty)$ are depicted in dotted line on Fig. 
\ref{figure5}(d) and the functions $\Lambda(0,T)$ and $\Lambda(\infty,T)$ are depicted in dotted line on Fig. 
\ref{figure5}(e). These figures show the graphs of functions $m\mapsto\Lambda(m,T)$ for various values of $T$, and the graphs of functions $T\mapsto\Lambda(m,T)$ for various values of $m$ and illustrate the convergence toward $\Lambda(m,0)$ and $\Lambda(m,\infty)$ as $T$ tends to 0 and $\infty$ and the convergence toward $\Lambda(0,T)$ and $\Lambda(\infty,T)$ as $m$ tends to 0 and $\infty$.

\begin{figure}[ht]
\begin{center}
\includegraphics[width=9cm,
viewport=170 500 430 660]{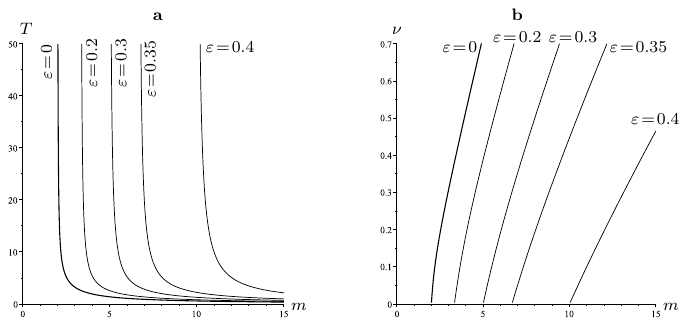}
\caption{The two patch system is given in Proposition \ref{Propab}, with $a=2$, $b=-1$.
{\bf a} Graphs of $m\mapsto T_c(m)$ with the indicated values of $\varepsilon$.
{\bf b} Graphs of $T\mapsto\nu_c(m)$ with the indicated values of $\varepsilon$. 
\label{figure6}}
\end{center}
\end{figure}

The previous example gives an illustration for Theorem \ref{DIDthmRed}. In the proof of this theorem, we have shown that for the migration matrix, which consists at each instant in migrating to the most unfavourable patch,  $\Lambda(\infty,T)=\sigma$. Therefore, if $\sigma<0$ and $m$ is large enough, then $\Lambda(m,T)<0$. 
However, it should be noted that, for DID to occur, migration needs not be to the worst-case patch only. Indeed, since the condition $\Lambda(m,T)<0$ is open, then the growth rate remains negative, even if small migration terms are added toward the favourable patches. The following proposition shows that migration terms from unfavourable patches to favourable patches can be quite large.  

\begin{propo}\label{Propab}
The system 
$$
\frac{dx}{dt}=\left\{
\begin{array}{l}
A_1x,\mbox{ if }t\in[0,T/2),\\  
A_2x,\mbox{ if }t\in[T/2,T),
\end{array}
\right.
\mbox{ with }
A_1=\left[
\begin{array}{cc}
a-m&\varepsilon m\\
m&b-\varepsilon m
\end{array}
\right],
~
A_2=\left[
\begin{array}{cc}
b-\varepsilon m&m\\
\varepsilon m&a-m
\end{array}
\right],
$$
where $a>0>b$ and
$\varepsilon\geq 0$, satisfies
$\Lambda(\infty,T)=\frac{b+\varepsilon a}{1+\varepsilon}.$
If $\varepsilon<-b/a<1$, then both patches are sources and DID occurs.
\end{propo}
\begin{proof}
The migration matrix s given by
$$L(\tau)=\left[\begin{array}{cc}
-1&\varepsilon\\
1&-\varepsilon
\end{array}\right] 
\mbox{ for }\tau\in[0,1/2)\quad\mbox{ and }\quad  
L(\tau)=\left[\begin{array}{cc}
-\varepsilon&1\\
\varepsilon&-1
\end{array}\right]
\mbox{ for }\tau\in[1/2,1).
$$ 
If $\varepsilon=0$, the formula $\Lambda(\infty,T)=b$ is a particular case of the formula for $\Lambda(\infty,T)$ given in Proposition \ref{Prop16}. 
If $\varepsilon>0$, the Perron-Frobenius vector 
$p(\tau)$ 
 of $L(\tau)$ is given by
$$
p_1(\tau)=\left\{
\begin{array}{l}
\frac{\varepsilon}{1+\varepsilon}\mbox{ if }\tau\in[0,1/2),\\
\frac{1}{1+\varepsilon}\mbox{ if }\tau\in[1/2,1),
\end{array}
\right.
\quad
p_2(\tau)=\left\{
\begin{array}{l}
\frac{1}{1+\varepsilon}\mbox{ if }\tau\in[0,1/2),\\
\frac{\varepsilon}{1+\varepsilon}\mbox{ if }\tau\in[1/2,1).
\end{array}
\right.
$$ 
 Using Theorem \ref{Prop15} (or \cite[Eq. (15)]{BLSS} since the migration matrix is irreducible for all $\tau$), 
$$\Lambda(\infty,T)=\overline{p_1r_1+p_2r_2}=\frac{b+\varepsilon a}{1+\varepsilon}.$$
Therefore, if $b+\varepsilon a<0$ and $m$ is large enough, then $\Lambda(m,T)<0$ for some $T$.
The condition $b+\varepsilon a<0$ is equivalent to the condition $\varepsilon<-b/a$. The condition $-b/a<1$ is equivalent to the condition $a+b>0$ which means that both patches are sources.
\end{proof}

The critical functions $T_c(m)$ and $\nu_c(m)$ such that decay occurs if $T>T_c(m)$ or, equivalently, $\nu<\nu_c(m)$, and growth occurs if $T<T_c(m)$ or, equivalently, $\nu>\nu_c(m)$, are shown in Fig. 
\ref{figure6}, for the system given in Proposition \ref{Propab}, with $a=2$ and $b=-1$. As you would expect, the larger $\varepsilon$ is, the smaller the area in which decay takes place.
According to Proposition \ref{Propab}, DID occurs for $0\leq \varepsilon<1/2$. 
Consequently, $\varepsilon$, which represents the migration term from the unfavourable patch to the favourable patch, can be at most half of 1, which represents the migration from the favourable patch to the unfavourable patch.

 \subsection{Three patches case}

We consider examples with three patches and three seasons, where the growth rate and the migration are constant during each 
season.  
\subsubsection{The threshold $\chi$ is positive, but DIG does not occur}\label{ExLobryNoDIG}

Consider the $T$-periodic linear system given by
\begin{equation}\label{eq47}
\frac{dx}{dt}=\left\{
\begin{array}{l}
A_1x,\mbox{ if }t\in[0,T/3),\\  
A_2x,\mbox{ if }t\in[T/3,2T/3),\\
A_3x,\mbox{ if }t\in[2T/3,T),
\end{array}
\right.
\end{equation}
where
$$
{A_1=\left(
\begin{array}{ccc}a&0&0\\
0&b-m&m\\
0&m&b-m
\end{array}
\right)
},
~ 
{A_2=
\left(
\begin{array}{ccc}
b-m&0&m\\
0&a&0\\
m&0&b-m
\end{array}
\right)
},
~
{A_3=
\left(
\begin{array}{ccc}
b-m&m&0\\
m&b-m&0\\
0&0&a
\end{array}
\right)
}.
$$

This system is shown in Fig. \ref{figure7}(a).
There are three patches, two with a growth rate equal to $b$ and the third with the growth rate
equal to $a>b$. The growth rate is indicated in the patch. Patch 1 is at bottom left, patch 2 at bottom right and patch 3 at top. There are 3 seasons. 
This system is the same as the three-patch model given in \cite[Fig. 8]{BLSS}. The migration is symmetric and is only between the patches where the growth rate is $b$. 
We saw in \cite[Section 4.5.2]{BLSS}, by numerical simulation, that DIG occurs for $a=1$ and $b=-0.8$, but does not occur when $a=1$ and $b=-1$, see \cite[Fig. 9]{BLSS}. The aim of this section is to use our theoretical results to better understand the behaviour of the system.

\begin{figure}[ht]
\begin{center}
\includegraphics[width=10cm,
viewport=150 490 440 680]{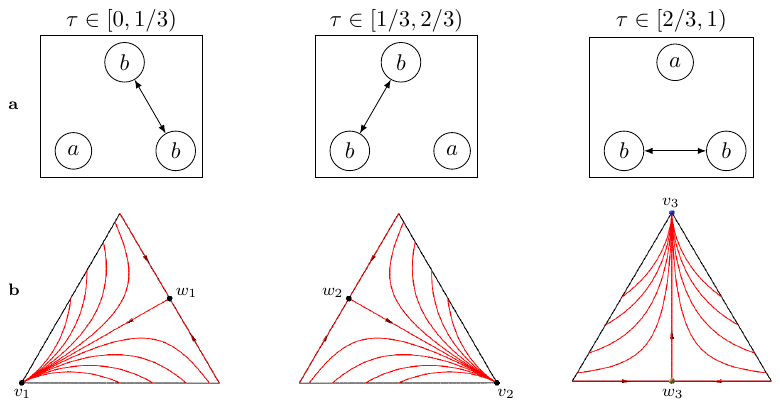}
\caption{The system is \eqref{eq47}. {\bf (a)} The migration is symmetrical between the unfavourable patches with groth rate $b\leq a$. {\bf (b)} The corresponding flow on the simplex. The condition H3.3 in Hypothesis \ref{H3} is not satisfied.
\label{figure7}}
\end{center}
\end{figure}

\begin{propo}\label{PropNoDIG}
The system \eqref{eq47}  admits a growth rate $\Lambda$ and
$$\Lambda(0,T)=\Lambda(m,0)=\frac{a+2b}{3}.$$
Moreover, Hypotheses \ref{H3} and \ref{H4} are not satisfied and we cannot use Theorems \ref{Prop14} and \ref{Prop15} to determine $\Lambda(m,\infty)$ and $\Lambda(\infty,T)$.
\end{propo}
\begin{proof}
The average of the migration matrix is
$$\overline{L}=
\left[\begin{array}{ccc}
-2/3&1/3&1/3\\
1/3&-2/3&1/3\\
1/3&1/3&-2/3
\end{array}\right] 
$$
It is irreducible. Therefore, the growth rate exists. Using Proposition \ref{Prop12m=0}, 
$$\Lambda(0,T)=\max\left(\overline{r}_1,\overline{r}_2,\overline{r}_3\right)=\frac{a+2b}{3}.$$ 
Moreover, the eigenvalues of $\overline{A}$ are $\frac{a+2b}{3}$, $\frac{a+2b}{3}-m$ and $\frac{a+2b}{3}-m$. Using Theorem  \ref{Prop13}, 
$$\Lambda(m,0)=\lambda_{max}\left(\overline{A}\right)=\frac{a+2b}{3}.$$ 

Let us look at why our Hypothesis \ref{H3}  is not satisfied.
The eigenvalues of $A_3$ are $a$, $b$, and $b-2m$. Since $a\geq b$, 
$$\lambda_{max}(A_3)=a.$$
The corresponding Perron-Frobenius vector is 
$v_3=(0,0,1)$. Note that the matrix has another nonnegative eigenvector, corresponding to its eigenvalue $b$, and given by $w_2=(1/2,1/2,0)$. The differential system associated to the matrix $A_3$ on the simplex $\Delta$, parametrized by $\theta_1$ and $\theta_2$,  is
$$
\begin{array}{lcl}
\frac{d\theta_1}{dt}&=&m(\theta_2-\theta_1)+(a-b)(\theta_1+\theta_2-1)\theta_1,\\[2mm]
\frac{d\theta_2}{dt}&=&m(\theta_1-\theta_2)+(a-b)(\theta_1+\theta_2-1)\theta_2.
\end{array}
$$
It admits $v_3$ and $w_3$ as equilibria, the first being stable and the second unstable (a saddle point), see Fig. \ref{figure7}(b). Note that the lines $\theta_3=0$ and $\theta_1=\theta_2$ are invariant by the flow. The basin of attraction of $v_3$ is the subset $\theta_3>0$ of the simplex.

Similarly, if $\tau\in[0,1/3)$, we see that, the differential equation on the simplex  admits $v_1=(1,0,0)$ and $w_1=(0,1/2,1/2)$ as equilibria, and,  when $\tau\in[1/3,2/3)$, its admits $v_2=(0,1,0)$ and $w_2=(1/2,0,1/2)$ as equilibria, see Fig. \ref{figure7}(b). 
Note that $v_1$ does not belong to the basin of attraction of $v_2$. Indeed, it is attracted by $w_2$. Similarly, 
$v_2$ does not belong to the basin of attraction of $v_3$, since it is attracted by $w_3$, and $v_3$ does not belong to the basin of attraction of $v_1$, since it is attracted by $w_1$. 
The condition {{(H3.3)}} in Hypothesis \ref{H3} is not satisfied. Therefore, Theorem \ref{Prop14} cannot be used to determine the limit $\Lambda(m,\infty)$.

\begin{figure}[ht]
\begin{center}
\includegraphics[width=9cm,
viewport=170 355 430 670]{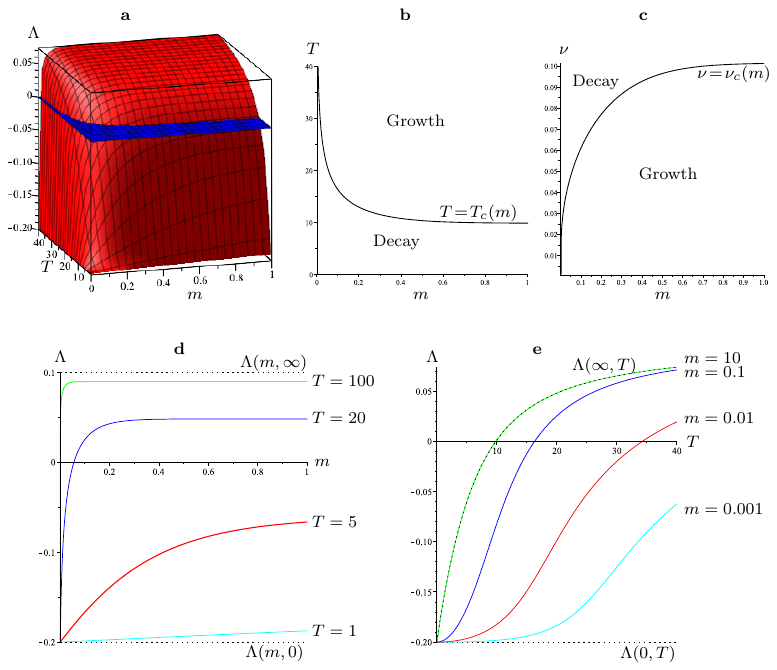}
\caption{{\bf a} The graph of $(m,T)\mapsto\Lambda(m,T)$. 
{\bf b} The set $\Lambda(m,T)=0$.
{\bf c} The set $\Lambda(m,1/\nu)=0$ in the $(m,\nu)$ parameter-plane.
{\bf d} Graphs of $m\mapsto\Lambda(m,T)$ with the indicated values of $T$.
{\bf e} Graphs of $T\mapsto\Lambda(m,T)$ with the indicated values of $m$. The limits are
$\Lambda(0,T)=\Lambda(m,0)=-0.2$, 
$\Lambda(m,\infty)=0.1$ and $\Lambda(m,\infty)$ is  given in Proposition \ref{PropES}
The system is \eqref{eq47}, with $a=1$ and $b=-0.8$.
\label{figure8}}
\end{center}
\end{figure}

The migration matrix admits $0$ as a double eigenvalue and $-2$ as a simple eigenvalue.
Its spectral abscissa is 0. It is not a simple eigenvalue. Therefore, Hypothesis \ref{H4} is not satisfied and Theorem \ref{Prop15} cannot be used to determine the limit  $\Lambda(\infty,T)$.
\end{proof}

We assume that $a>0>b$ and $a+2b<0$. Therefore $\overline{r}_1=\overline{r}_2=\overline{r}_3=\frac{a+2b}{3}<0$ and $\chi=a>0$. Does the DIG phenomenon occurs for this system ?  
In fact, DIG occurs when $a+b> 0$ and it is conjectured that it does not occur when $a+b\leq 0$.

\begin{propo}\label{PropES}
Let $\alpha=e^{\frac{a-b}{3}T}$ and
$\beta=e^{\frac{b-a}{3}T}$. We have
$$\Lambda(\infty,T)=\frac{a+2b}{3}+
\frac{1}{T}\ln\left(\frac{3/2+
\beta/4+(1/4)\sqrt{36+\beta^2+12\beta+32\alpha}}{4}\right),
$$
We also have $\Lambda(m,\infty)=\frac{a+b}{2}.
$
Therefore, if $-2b>a>-b$, DIG occurs.
\end{propo}
\begin{proof}
The proof is given in Section \ref{ProofPropES}.
\end{proof}

We illustrate our results with $a=1$ and $b=-0.8$, where its was numerically observed in \cite[Fig. 9]{BLSS} that DIG occurs.
We show in 
Fig. \ref{figure8}(a) the graph of the growth rate $\Lambda(m,T)$ and in Fig. 
\ref{figure8}(b,c) the sets where $\Lambda(m,T)=0$ or $\Lambda(m,1/\nu)=0$. The figure shows that there exists a critical curve $T=T_c(m)$,  such that $T_c(0)=T_c(m^*)=\infty$ and growth occurs if and only if $T>T_c(m)$, or $\nu<\nu_c(m)$. Note that the critical curve $v=v_c(m)$ has asymptotic behavior, of the form $m\sim e^{-k/\nu}$ near the origin.  
Using the formulas in Propositions \ref{PropNoDIG} and \ref{PropES} we obtains the limits $\Lambda(0,T)$, $\Lambda(m,0)$, $\Lambda(\infty,T)$, $\Lambda(m,\infty)$, that are plotted in dotted lines 
 in Fig. \ref{figure8}. This figure shows that the functions $m\mapsto\Lambda(m,T)$ are strictly increasing for each $T>0$ and the functions $T\mapsto\Lambda(m,T)$ also are strictly increasing for each $m>0$ fixed. 

\begin{conj} 
 We conjecture that
for all $m>0$ and $T>0$, 
$\Lambda(m,T)\leq \frac{a+b}{2}$.
\end{conj}

If this conjecture is true, DIG occurs if and only if $a+b>0$.
For complements on this example, in particular on the fact that $m$ becomes exponentially small in $1/\nu$ near the origin, see  \cite{Circuits}.

\subsubsection{Unidirectional migration}\label{atob}

We consider again the linear system with three seasons
\begin{equation}\label{eq47ab}
\frac{dx}{dt}=\left\{
\begin{array}{l}
A_1x,\mbox{ if }t\in[0,T/3),\\  
A_2x,\mbox{ if }t\in[T/3,2T/3),\\
A_3x,\mbox{ if }t\in[2T/3,T),
\end{array}
\right.
\end{equation}
where $A_1$, $A_2$ and $A_3$ are given now by
\begin{equation}\label{A1A2A3}
{A_1=\left(
\begin{array}{ccc}a-m&0&0\\
m&b&0\\
0&0&b
\end{array}
\right)
},
~ 
{A_2=
\left(
\begin{array}{ccc}
b&0&m\\
0&b&0\\
0&0&a-m
\end{array}
\right)
},
~
{A_3=
\left(
\begin{array}{ccc}
b&0&0\\
0&a-m&0\\
0&m&b
\end{array}
\right)
}.
\end{equation}

 \begin{figure}[ht]
\begin{center}
\includegraphics[width=10cm,
viewport=150 490 440 680]{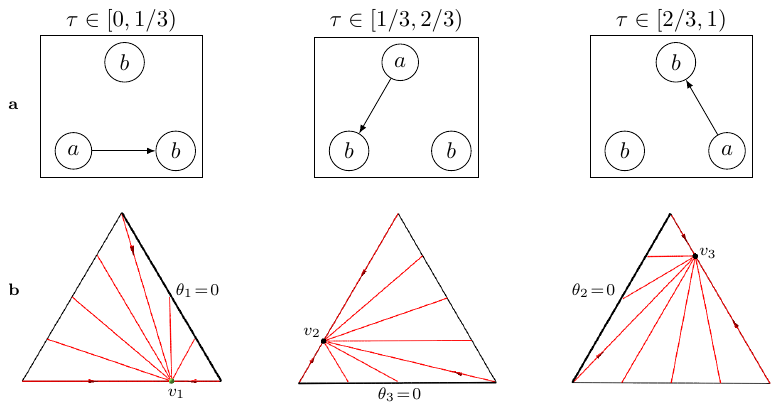}
\caption{The system is \eqref{eq47ab}. {\bf (a)} The migration is from an patch with growth equal to $a$ to a patch with growth equal to $b\leq a$. {\bf (b)} The corresponding flow on the simplex. The condition H3.3 in Hypothesis \ref{H3} is not satisfied.
\label{figure9}}
\end{center}
\end{figure}

\begin{figure}[ht]
\begin{center}
\includegraphics[width=10cm,
viewport=150 380 440 680]{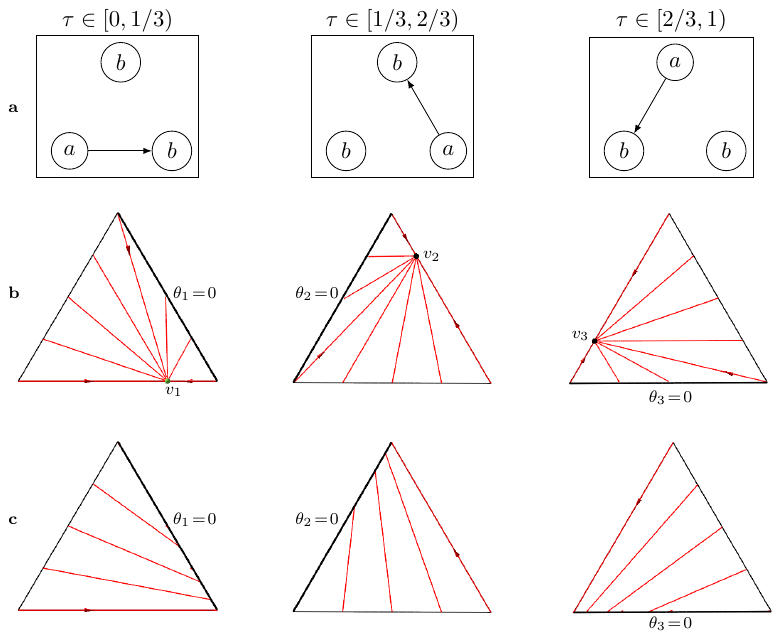}
\caption{The system is \eqref{eq47ab1}. {\bf (a)} The migration is from an patch with growth equal to $a$ to a patch with growth equal to $b\leq a$ and is obtained from  Fig. \ref{figure9}(a) by swapping seasons 2 and 3. {\bf (b)} The corresponding flow on the simplex in the case $0<m<a-b$. Hypothesis \ref{H3} is satisfied.
{\bf (c)} The flow on the simplex in the 
case $m\geq a-b$: the condition H3.1 in 
Hypothesis \ref{H3} is not satisfied.
\label{figure10}}
\end{center}
\end{figure}

This system is shown in Fig. \ref{figure9}(a). There is only one migration per season from a patch with growth equal to  $a$, to a patch with growth equal to $b$. 

We also consider the system 
\begin{equation}\label{eq47ab1}
\frac{dx}{dt}=\left\{
\begin{array}{l}
A_1x,\mbox{ if }t\in[0,T/3),\\  
A_3x,\mbox{ if }t\in[T/3,2T/3),\\
A_2x,\mbox{ if }t\in[2T/3,T),
\end{array}
\right.
\end{equation}
where $A_1$, $A_2$ and $A_3$ are also given by \eqref{A1A2A3}.  This system corresponds the the system shown in Fig. \ref{figure10}(a). Note that in Fig. \ref{figure10}(a), seasons 2 and 3 are swapped.


\begin{propo}\label{Propatob}
For the systems \eqref{eq47ab}, and \eqref{eq47ab1}
the growth rate exits  and satisfies
$$\Lambda(0,T)=\Lambda(m,0)=\frac{a+2b}{3}.$$
Moreover, Hypothesis \ref{H4} is not satisfied and we cannot use Theorem \ref{Prop15} to determine $\Lambda(\infty,T)$.
For the system \eqref{eq47ab} Hypothesis \ref{H3} is not satisfied and we cannot use Theorem \ref{Prop14}  to determine $\Lambda(m,\infty)$. For the system \eqref{eq47ab1}  Hypothesis \ref{H3} is satisfied for $0<m<b-a$ and
$\Lambda(m,\infty)=a-m.$
For $m\geq b-a$, Hypothesis \ref{H3} is not satisfied and we cannot use Theorem \ref{Prop14}  to determine $\Lambda(m,\infty)$.
\end{propo}

\begin{conj}\label{Conj1}
For the system \eqref{eq47ab} we conjecture that 
for all $T>0$, $\Lambda(\infty,T)=b$ and 
for all $m>0$, $\Lambda(m,\infty)=(a+b-m)/2$.  
\end{conj}

\begin{conj}\label{Conj2}
For the system \eqref{eq47ab1} we conjecture that 
for all $T>0$, $\Lambda(\infty,T)=b$ and
for all $ m\geq a-b$, $
\Lambda(m,\infty)=b$.
\end{conj}

\begin{proof}
The average of the migration matrix is
$$\overline{L}=\frac{1}{3}
\left[\begin{array}{ccc}
-1&0&1\\
1&-1&0\\
0&1&-1
\end{array}\right] 
$$
It is irreducible. Therefore, the growth rate exists and, using Proposition \ref{Prop12m=0}, 
$$\Lambda(0,T)=\max\left(\overline{r}_1,\overline{r}_2,\overline{r}_3\right)=\frac{a+2b}{3}.$$ 
The average of the matrix $A$ is
$$\overline{A}=\frac{1}{3}
\left[\begin{array}{ccc}
a+2b-m&0&m\\
m&a+2b&0\\
0&m&a+2b-m
\end{array}\right] 
$$
The spectral abscissa of $\overline{A}$ is $\frac{a+2b}{3}$. Therefore, using Theorem \ref{Prop13},
$$\Lambda(m,0)=\lambda_{max}\left(\overline{A}\right)=\frac{a+2b}{3}.$$

Moreover, in each season the migration matrix has eigenvalues $0$, $0$, and $-1$. Therefore its spectral abscissa 0 is not a simple eigenvalue and Hypothesis \ref{H4} is not satisfied.

Let us look at why our Hypothesis \ref{H3}  is not satisfied for the system \eqref{eq47ab}.
The eigenvalues of $A_1$ are $a-m$ and $b$ (double). The spectral abscissa of $A_1$ is 
$$\lambda_{max}(A_1)=\left\{
\begin{array}{lcl}
a-m&\mbox{if}& 0<m<a-b,\\
b&\mbox{if} &m\geq a-b.
\end{array}
\right.$$

Let us first consider the case $m<a-b$.
The corresponding Perron-Frobenius vector is 
$v_1=\left(\frac{a-b-m}{a-b},\frac{m}{a-b},0\right)$. Note that the eigenvalue $b$ admits $(0,0,1)$ and $(0,1,0)$ as eigenvectors. The differential system associated to the matrix $A_1$ on the simplex $\Delta$, parametrized by $\theta_1$ and $\theta_2$,  is
$$
\begin{array}{lcl}
\frac{d\theta_1}{dt}&=&-\theta_1((a-b)(\theta_1-1)+m),\\[2mm]
\frac{d\theta_2}{dt}&=&-\theta_1((a-b)\theta_2-m).
\end{array}
$$
It admits $v_1$ as a stable equilibrium. It also admits the set $\theta_1=0$ of non isolated equilibria, which are the eigenvectors corresponding to the eigenvalue $b$, see Figure \ref{figure9}(b). Note that the orbits are straight lines passing through $v_1$. The basin of attraction of $v_1$ is the subset $\theta_1>0$ of $\Delta$.

When $\tau\in[1/3,2/3)$, we see that the differential equation on the simplex  has a set of non-isolated equilibria, given by $\theta_3=0$ and the equilibrium 
$v_2=\left(\frac{m}{a-b},0,\frac{a-b-m}{a-b}\right)$,
whose basin of attraction is is the subset $\theta_3>0$ of $\Delta$.
Similarly, when $\tau\in[2/3,1)$, we see that, the differential equation on the simplex admits
has a set of non-isolated equilibria, given by $\theta_2=0$ and the equilibrium 
$v_2=\left(0,\frac{m}{a-b},\frac{a-b-m}{a-b}\right)$,
whose basin of attraction is is the subset $\theta_2>0$ of $\Delta$.

Note that $v_3=v(0-0)$ does not belong to the basin of attraction of $v_1=v(0)$. Indeed, it belongs to the invariant set $\theta_1=0$. Similarly, 
$v_1=v(1/3-0)$ does not belong to the basin of attraction of $v_2=v(1/3)$ and $v_2=v(2/3-0)$ does not belong to the basin of attraction of $v_3=v(2/3)$.
Therefore, Hypothesis H3 is not satisfied and Theorem \ref{Prop14} cannot be used to determine the limit $\Lambda(m,\infty)$.

\begin{figure}[ht]
\begin{center}
\includegraphics[width=9cm,
viewport=170 355 430 670]{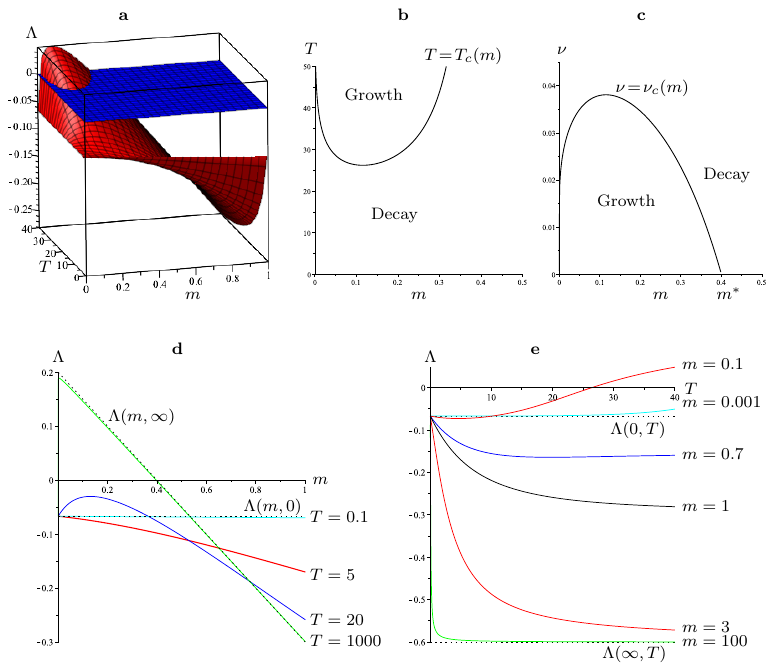}
\caption{
The system is \eqref{eq47ab}, with $a=1$ and $b=-0.6$.
{\bf a} The graph of $(m,T)\mapsto\Lambda(m,T)$. 
{\bf b} The set $\Lambda(m,T)=0$.
{\bf c} The set $\Lambda(m,1/\nu)=0$ in the $(m,\nu)$ parameter-plane.
{\bf d} Graphs of $m\mapsto\Lambda(m,T)$ with the indicated values of $T$.
{\bf e} Graphs of $T\mapsto\Lambda(m,T)$ with the indicated values of $m$.  
$\Lambda(0,T)=\Lambda(m,0)=-0.2/3$ are given in Proposition \ref{Propatob}. The limits 
$\Lambda(\infty,T)=-0.6$ and $\Lambda(m,\infty)=(0.4-m)/2$ support Conjecture \ref{Conj1}.
\label{figure11}}
\end{center}
\end{figure}

\begin{figure}[ht]
\begin{center}
\includegraphics[width=9cm,
viewport=170 355 430 670]{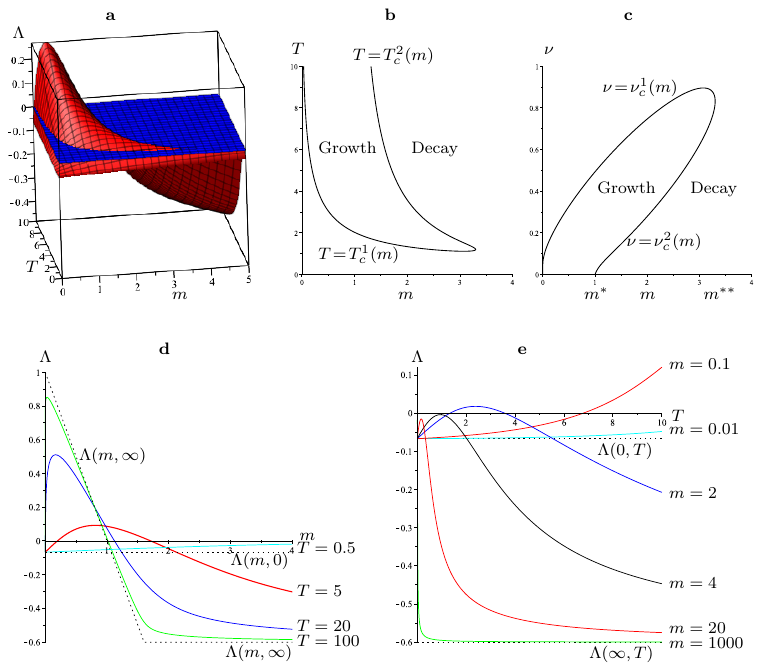}
\caption{
The system is \eqref{eq47ab1}, with $a=1$ and $b=-0.6$.
{\bf a} The graph of $(m,T)\mapsto\Lambda(m,T)$. 
{\bf b} The set $\Lambda(m,T)=0$.
{\bf c} The set $\Lambda(m,1/\nu)=0$ in the $(m,\nu)$ parameter-plane.
{\bf d} Graphs of $m\mapsto\Lambda(m,T)$ with the indicated values of $T$.
{\bf e} Graphs of $T\mapsto\Lambda(m,T)$ with the indicated values of $m$.
$\Lambda(0,T)=\Lambda(m,0)=-0.2/3$ and $\Lambda(m,\infty)=1-m$, for $m<1.6$ are given in Proposition \ref{Propatob}. The limits 
$\Lambda(\infty,T)=-0.6$ and $\Lambda(m,\infty)=-0.6$, for $m\geq 1.6$ support Conjecture \ref{Conj2}.
\label{figure12}}
\end{center}
\end{figure}

For the system \eqref{eq47ab1}, seasons 2 and 3 are swapped. The flow on the simplex is shown in Fig. \ref{figure10}(b). Hypothesis \ref{H3} satisfied. Indeed, $v_2=v(0-0)$ belongs to the basin of attraction of $v_1=v(0)$, 
$v_1=v(1/3-0)$ belongs to the basin of attraction of $v_3=v(1/3)$ and $v_3=v(2/3-0)$ belongs to the basin of attraction of $v_2=v(2/3)$.
Therefore,  using Theorem \ref{Prop14},
$$\Lambda(m,\infty)=\frac{\lambda_{max}(A_1)+\lambda_{max}(A_2)+\lambda_{max}(A_3)}{3}=\frac{a-m+a-m+a-m}{3}=a-m.$$

If $m\geq a-b$ the spectral abscissa is not a simple eigenvalue of $A_1$ and the corresponding eigenvectors are stable but not asymptotically stable, see Fig. \ref{figure10}(c). Therefore, the condition H3.1 in Hypothesis \ref{H3}  is not satisfied. 
 \end{proof}

We illustrate our results in  the case of the system shown in  
Fig. \ref{figure9}, with $a=1$ and $b=-0.6$.
We show in 
Fig. \ref{figure11} the graph of  $\Lambda(m,T)$. The figure also shows the sets where $\Lambda(m,T)=0$ or $\Lambda(m,1/\nu)=0$: there exists a critical curve $T=T_c(m)$,  such that $T_c(0)=T_c(m^*)=\infty$, with $m^*=a+b=0.4$, and growth occurs if and only if $T>T_c(m)$, or $\nu<\nu_c(m)$. Note that the critical curve $v=v_c(m)$ has asymptotic behavior, of the form $m\sim e^{-k/\nu}$ near the origin. The limits $\Lambda(0,T)$, $\Lambda(m,0)$, $\Lambda(\infty,T)$ and $\Lambda(m,\infty)$  are plotted in dotted lines 
 in Fig. \ref{figure11}(d,e), as well as the graphs of the functions $m\mapsto\Lambda(m,T)$ and $T\mapsto\Lambda(m,T)$. 
Using the formulas in Proposition \ref{Propatob} we obtain the limits $\Lambda(0,T)$ and $\Lambda(m,0)$. We cannot use 
Theorems \ref{Prop14}  
and \ref{Prop15} to determine $\Lambda(m,\infty)$ and $\Lambda(\infty,T)$.

We illustrate our results in  the case of the system shown in  
Fig. \ref{figure10}, with $a=1$ and $b=-0.6$.
We show in 
Fig. \ref{figure12} the graph of  $\Lambda(m,T)$. The figure also shows the sets where $\Lambda(m,T)=0$ or $\Lambda(m,1/\nu)=0$: there exists $m^{**}>m^*$ and two functions 
$$T_c^1:(0,m^{**})\to(0,+\infty),
\quad
T_c^2:(m^*,m^{**})\to(0,+\infty),
$$  
such that $T_c^1(0)=T_c^2(m^*)=+\infty$ and $T_c^1(m^{**})=
T_c^2(m^{**})$, and growth occurs if and only if and only $T_c^1(m)>T>T_c(m)$, or equivalently, $\nu_c^1(m)<v<v_c^2(m)$, where $\nu=1/T$ and 
$\nu_c^i(m)=1/T_c^i(m)$, for $i=1,2$. Note $m^*=1$ is the value of $m$ for which $\Lambda(m\infty)$.  Observe that the critical curve $\nu=\nu_c^1(m)$ is tangent to the $\nu$-axis ate the origin. 

The limits $\Lambda(0,T)$, $\Lambda(m,0)$, $\Lambda(\infty,T)$ and $\Lambda(m,\infty)$  are plotted in dotted lines 
 in Fig. \ref{figure12}(d,e), as well as the graphs of the functions $m\mapsto\Lambda(m,T)$ and $T\mapsto\Lambda(m,T)$. 
Using the formulas in Proposition \ref{Propatob} we obtain the limits $\Lambda(0,T)$, $\Lambda(m,0)$ and $\Lambda(m,\infty)$ for $0<m<a-b$. 
Theorem \ref{Prop14} cannot be used to determine  $\Lambda(m,\infty)$ for $m\geq a-b$. Theorem \ref{Prop15} cannot be used to determine  and $\Lambda(\infty,T)$.

\section{Discussion}\label{Discussion}

We have considered the $T$-periodic piecewise continuous linear differential system \eqref{eq1}
representing $n$  
populations of sizes $x_i(t)$, inhabiting $n$ patches, and subject to  $T$-periodic piecewise continuous local growth rates $r_i(t)$ ($1\leq i\leq n$), and migration rates $m\ell_{ij}(t)\geq 0$,  from patch $j$ to patch $i$,  where the parameter $m\geq 0$ measures the strength of migration, and 
the numbers 
$
\ell_{ij}(t)\geq 0$, $i\neq j$,
 encode the relative rates of dispersal among different patches. We extended the results in \cite{BLSS}, obtained in the case where the migration matrix $L(t)=(\ell_{ij}(t))$ is irreducible for all $t$, to the more general case where only the averaged matrix $\overline{L}$ is assumed to be irreducible.

We proved that, as soon as $m$ is positive,
$$\lim_{t\to\infty}\frac{1}{t}\ln(x_i(t))=\Lambda(m,T):=\frac{1}{T}\ln(\mu(m,T)),$$
where $\mu(m,T)$ is the Perron-Frobenius root of the monodromy matrix associated to \eqref{eq1}. Indeed, the irreducibility of $\overline{L}$ implies that this matrix has non negative entries and is irreducible.  Hence the growth rate is the same on every patch. 

Following \cite{BLSS}, our proofs rely mostly on the reduction of the system on the simplex $\Delta$. Instead of considering the size $x_i$ of the population on each patch we consider the total population $\rho  = \sum_{i=1}^n x_i$ and the proportion  $\theta_i = x_i/\rho$ on each patch. In these new variables $(\rho,\theta)$ the system has nice properties. It turns out that the system of  $\theta$ variables  is non linear, but independent of $\rho$.
We  prove then  that it has a globally asymptotically stable periodic solution $\theta^*$  from which we deduce the existence of the growth rate $\Lambda(m,T)$, and an integral expression of it, using the periodic function $\theta^*$. Moreover, on this $\theta$ system we can apply averaging, from which we can deduce our small $T$ asymptotics  of $\Lambda(m,T)$, and Tikhonov's theorem from which we deduce our large $m$ or $T$  asymptotics of $\Lambda(m,T)$.

In the case considered in \cite{BLSS}, 
the
the irreducibility of the matrix $L(\tau)$ implies that of $A(\tau)=R(\tau)+mL(\tau)$ and therefore the existence of its spectral abscissa $\lambda_{max}(A(\tau))$, and the fact that the corresponding Perron-Frobenius vector $v(\tau)$ is a GAS equilibrium of the differential equation 
\eqref{eqsimplextau} on the simplex. 
In the more general case of this paper, this property is not always satisfied and we must add Hypothesis \ref{H3}, saying that $v(\tau)$ exists and is a an asymptotically stable equilibrium of the differential equation
\eqref{eqsimplextau}
and has a basin of attraction which is uniform with respect to the parameter $\tau\in[0,1]$.
We then derive formula \eqref{mT=infini} giving the limit $\Lambda(m,\infty)$ of $\Lambda(m,T)$ when $T$ tends to infinity and the formula \eqref{m=infini,T=infini} giving the limit of $\Lambda(m,\infty)$ when $m$ tends to 0.

In the case considered in \cite{BLSS}, 
the irreducibility of the matrix $L(\tau)$ implies the existence of its Perron-Frobenius vector $p(\tau)$ and the fact that $p(\tau)$ is a GAS equilibrium of the differential equation \eqref{eqsimplextauL} on the simplex. In the more general case of this paper, this property is not always satisfied and we must add Hypothesis \ref{H4}, saying that $p(\tau)$ exists and is a an asymptotically stable equilibrium of the differential equation \eqref{eqsimplextauL} and admits a basin of attraction which is uniform in $\tau\in[0,1]$.  We then derive formula \eqref{m=infini} giving the limit of $\Lambda(m,T)$ when $m$ tends to infinity.

The formula \eqref{m=0T=infini} plays a major role in the study of the DIG phenomenon. Indeed, this formula shows that the number $\chi$ is the supremum of $\Lambda(m,T)$ and explains why DIG occurs if and only if $\chi$ is positive. 
Also, the formula \eqref{m=infini} plays a major role in the study of the DID phenomenon. Indeed, this formula shows that for a class of migration matrices, the number $\sigma$ is the infimum of $\Lambda(m,T)$ and explains why DID occurs if and only if $\sigma$ is negative. 

{Formula \eqref{m=0T=infini}, giving the limit when $m\to 0$ of the limit when $T\to\infty$ of $\Lambda(m,T)$, and Formula \eqref{m=infini}, 
giving the limit when $m\to \infty$ of $\Lambda(m,T)$,
which play a major role in the study of the DIG and DID phenomenon, require addition} 
of Hypotheses \ref{H3} and \ref{H4}, respectively. We presented in Section \ref{Numerical} several examples of how these hypotheses can be verified, as well as a case {in which they do not hold}. 
The study of the asymptotics of $\Lambda(m,T)$ when $T$ is large or $m$ is large, in the case where Hypotheses \ref{H3} or \ref{H4} are not satisfied is a major open question that will be investigated in future work.

Most of the results of \cite{BLSS} remain valid in the stochastic case, 
see \cite[Theorem 13]{BLSS}. The existence of the growth rate $\Lambda(T)$ requires only the assumption that $\overline{A}$ is irreducible \cite{BLSSarXiv}. However, in the stochastic case, the formula that gives the limit of $\Lambda(T)$ when $T$ tends to infinity is only proved in the case where for all $t$, $A(t)$ is irreducible \cite[Proposition 1.10]{BLSSarXiv}.  The validity of this formula in the more general case of Hypothesis \ref{H3}, considered in this article, is an open problem.

We considered the surprising \emph{dispersal induced growth (DIG)} phenomenon, where the populations persist and grow exponentially, despite the fact that all patches are sinks (i.e. $\overline{r}_i<0$, for all $i$) when there is no dispersal between them. We also considered the surprising \emph{dispersal induced decay (DID)} phenomenon, where the populations decay exponentially, despite the fact that all patches are sources (i.e. $\overline{r}_i>0$, for all $i$) when there is no dispersal between them. For this purpose we considered the numbers
$
\sigma=\int_0^1{\min_{1\leq i\leq n}r_i(t)}dt$ and 
$\chi=\int_0^1{\max_{1\leq i\leq n}r_i(t)}dt.
$

Consider the idealized habitat, called the \emph{ideal best habitat},  whose growth rate at  any time is that of the habitat with maximal growth at this time. Hence, $\chi$ is the average growth rate in this idealized  habitat. 
If the population does not grow exponentially in the idealized best habitat (i.e. if $\chi\leq 0$), then from Theorem \ref{Bornes2} we deduce that DIG does not occur. Moreover, thanks to 
Theorem \ref{DIGthmRed}
the population can survive if and only it would survive in the idealized best habitat.

Similarly, consider the idealized habitat, called the \emph{ideal worst habitat},  whose growth rate at any time is that of the habitat with minimal growth at this time. Hence, $\sigma$,  is the average growth rate in this idealized habitat. 
If the population {does not become extinct} in the idealized worst habitat (i.e. if $\sigma\geq 0$), then from Theorem \ref{Bornes2} we deduce that DID does not occur. Moreover, thanks to 
Theorem \ref{DIDthmRed}, if
the population {become extinct} in the idealized worst habitat, then there exist migration matrices for which it {become extinct} in the real environment with dispersion. 

Note that the DID phenomenon is less biologically significant than that of DIG. Whereas DIG is valid under very general conditions, including when dispersal is independent of time, DID requires a judicious choice of dispersal rates so that migration at any given time is biased towards the `wrong' patches. It is reasonable to assume that such migration behavior will not be selected by evolution.
However, to show the practical importance of the DID phenomenon, we can look at it from the epidemiological point of view of disease and pest control: we are looking for disease decrease, and so the fact that we need to choose migrations wisely to achieve the decrease can be interesting.

The possibility that $L(\tau)$ is not irreducible for all $\tau$  is not a simple desire for mathematical generality. Indeed, the assumption that the migration matrix is irreducible is certainly not realized in many real systems. Our study recovers in particular the case of two patches with two seasons and, during the
first season, there is migration from patch one to patch two and conversely from patch
two to one during season two, like migrating birds do between places in north or south. 
In this case the migration matrix $L(\tau)$ is not irreducible but DIG or DID can be observed as we have shown in Section~\ref{Numerical}. 
Think of birds migrating seasonally from Arctic to subtropical areas, or from Southern to Northern hemisphere; it seems likely that they would get extinct if they were fixed in either habitat, but they survive well with migrations. This example of migratory birds, although it seems to make the DIG phenomenon fairly intuitive and not surprising, doesn't fully explain all its subtleties. In fact, the examples we have highlighted show that when migration takes place solely towards the most favourable site, the DIG phenomenon occurs for all sufficiently large $m$ and $T$ (see Section \ref{MFP}), and that it also can occur even when migration takes place solely towards the least favourable site (see Section \ref{MUP}).

\appendix
\section{Proofs}\label{Proofs}

The proofs in the more general context of this paper follow the same steps as the proofs in \cite{BLSS}. The main tool in \cite{BLSS} was Perron's theorem applied to the monodromy matrix, which is positive because, in \cite{BLSS}, $A(t)$ was assumed to be irreducible for all $t\in [0,1]$. In the more general context of this paper where this assumption is replaced by the irreducibility of the average matrix $\overline{A}$, the main tool is Perron-Frobenius's theorem applied to the monodromy matrix, which is nonnegative and irreducible according to Lemma \ref{lem1}. For the sake of completeness, we give a sketch of the proofs, and we refer the reader to \cite{BLSS} for the  details.

\subsection{Proof of Theorem \ref{thm2}}
\label{ProofThm2}

Recall that the solution $x(t,x_0)$ to \eqref{eq3} such that $x(0,x_0)=x_0$ writes 
\begin{equation}\label{flot}
x(t,x_0)=X(t)x_0
\end{equation}
where $X(t)$ is the solution to the matrix valued differential equation \eqref{MVDE}.
The flow \eqref{flot} of \eqref{eq3} induces a flow on $\Delta$, given by
\begin{equation}\label{flotPsi}
\Psi(t,\theta)=\frac{X(t)\theta}{\langle X(t)\theta,{\bf 1}\rangle}.
\end{equation}
Let $\Phi:=X(1)$ be the monodromy matrix, $\mu$ its Perron-Frobenius root and $\pi$ the corresponding
{Perron-Frobenius vector}. 
Using $\Phi\pi=\mu\pi$,
$\langle \pi,{\bf 1}\rangle=1$,
 and \eqref{flotPsi},
$$\Psi(1,\pi)=\frac{\Phi\pi}{\langle \Phi\pi,{\bf 1}\rangle}=\frac{\mu\pi}{\mu\langle \pi,{\bf 1}\rangle}=\pi. $$ 
 Therefore $\pi$ is a fixed point of $\Psi(1,\theta)$, so that  $\Psi(t,\pi)$ is a $1$-periodic orbit in $\Delta$. The global stability of this orbit follows from the Perron-Frobenius projector formula
\begin{equation}\label{Conv}
\lim_{k\to\infty}\frac{\Phi^kx}{\langle\Phi^kx,{\bf 1}\rangle}= \pi,
\end{equation} 

For any $\theta_0\in\Delta$, the solution $\theta(t)$ of \eqref{eqtheta} with initial condition $\theta(0)=\theta_0$ is given by 
$\theta(t,\theta_0)=\Psi(t,\theta_0)$, where $\Psi$ is given by \eqref{flotPsi}. Using \eqref{eqrho}, and $x=\rho\theta$,
\begin{equation}\label{x(t,x0)}
x(t,x_0)=\theta\left(t,{x_0}/{\rho_0}\right)\rho_0e^{\int_0^t\langle A(s)\theta\left(s,{x_0}/{\rho_0}\right),{\bf 1}\rangle ds},
\end{equation}
where $\rho_0=\langle x_0,{\bf 1}\rangle$.
Using \eqref{x(t,x0)},  the Lyapunov exponent of the components of any solution $x(t,x_0)$ can be computed as follows:
\begin{align*}
\lim_{t\to\infty}\frac{1}{t}\ln(x_i(t,x_0))=\lim_{k\to\infty}\frac{1}{k}\ln(x_i(k,x_0))=
\lim_{k\to\infty}\frac{1}{k}\left[
\ln(\theta_i(k,x_0/\rho_0)\rho_0)+
\int_0^{k}U(s) ds\right],
\end{align*}
where $U(s)=\langle A(s)\theta\left(s,{x_0}/{\rho_0}\right),{\bf 1}\rangle$. Using the global asymptotic stability of the periodic orbit $\theta^*(t):=\Psi(t,\pi)$, 
\begin{align*}
\lim_{t\to\infty}\frac{1}{t}\ln(x_i(t,x_0))&=
\lim_{k\to\infty}\frac{1}{k}\int_{k_1}^{k}
\langle A(s)\theta^*(s),{\bf 1}\rangle ds\\
&=\lim_{k\to\infty}\frac{k-k_1}{k}\int_{0}^{1}
\langle A(s)\theta^*(s),{\bf 1}\rangle ds
=\int_0^1\langle A(t)\theta^*(s),{\bf 1}\rangle ds.
\end{align*}
For the details, we refer the reader to the proof of \cite[Theorem 25]{BLSS}.
This proves the first equality in \eqref{Lambda=Lambda[rho]1}. 
 Let $x(t)=X(t)\pi$ be the solution of  \eqref{eq3}, such that $x(0)=\pi$. Since  $x(1)=\Phi\pi=\mu\pi$, we have $x(k)=\Phi^k\pi=\mu^k\pi$. Hence
 $$\lim_{k\to\infty}\frac{1}{k}\ln(x_i(k))=\ln(\mu)=\Lambda,$$
 which proves the second equality in \eqref{Lambda=Lambda[rho]1}.

 \subsection{Proof of Theorem \ref{Prop13}}
\label{ProofProp13}
The proof is the same as the proof in \cite[Eq. (12)]{BLSS}.
Indeed, the proof  in \cite{BLSS} only uses the averaging theorem, see \cite[Theorem 17]{BLSS}, and the fact the matrix $\overline{A}$ is irreducible. 

Since the matrix  
$\overline{A}$ is irreducible, its 
Perron-Frobenius vector $w=(w_1,\cdots,w_n)^\top$, corresponding to its spectral abscissa $\lambda_{max}(\overline{A})$, is a globally asymptotically stable equilibrium of the averaged equation of \eqref{eqtheta} on the simplex $\Delta$ 
$$
\begin{array}{lcl}
\frac{d\theta}{dt}&=&\overline{A}\theta- \langle \overline{A}\theta,
{\bf 1}\rangle \theta.
\end{array}
$$
Using the averaging theorem, we deduce that, as $T\to 0$, the $T$-periodic solution $\theta^*(t,T)$ of \eqref{eqtheta} converges toward $w$.
Hence,  using \eqref{Lambda=Lambda[rho]1}, as $T\to 0$, 
$$\Lambda(T)=\int_0^{1}
\langle A(\tau)\theta^*(T\tau,T),{\bf 1}\rangle 
d\tau
=
\int_0^{1}
\langle A(\tau)w,{\bf 1}\rangle 
d\tau+o(1)=\lambda_{max}\left(\overline{A}\right)+o(1).
$$
For the details, we refer the reader to  \cite[Section 5.3]{BLSS}. This proves  \eqref{T=0}.
 
\subsection{Proof of Theorem \ref{Prop14}}
\label{ProofProp14}

The proof follows the same steps as the proof of \cite[Eq. (13)]{BLSS}.
Indeed, the result in \cite{BLSS} only uses
Tikhonov's theorem on singular perturbations, see \cite[Proposition 27]{BLSS}, and the fact that the Perron-Frobenius vector $v(\tau)$ of $A(\tau)$ is globally asymptotically stable for  \eqref{eqsimplextau}. Note that, in the case where $A(\tau)$ is irreducible for all $\tau\in[0,1]$, we have $v(\tau)\gg 0$, and its global asymptotic stability in the simplex $\Delta$ is guaranteed, see \cite[Proposition 24]{BLSS}. In the more general case where only the average matrix $\overline{A}$ is assumed to be irreducible, we need to introduce Hypothesis \ref{H3}.

We use the change of variable $\tau=t/T$ and $\eta(\tau)=\theta(T\tau)$. The equation \eqref{eqtheta} becomes
\begin{equation}\label{ztau}
\frac{1}{T}\frac{d\eta}{d\tau}=
A(\tau)\eta- \langle A(\tau )\eta,1\rangle\eta
\end{equation}

Theorem \ref{ThmThik}, in Appendix \ref{SingPert}, applies to the equation \eqref{ztau} 
which can then be written, noting $\varepsilon=1/T$, as
\begin{equation*}
\varepsilon\frac{d\eta}{d\tau}=
A(\tau)\eta- \langle A(\tau )\eta,1\rangle\eta
\end{equation*}
Using the fast time $t=\varepsilon\tau$, the fast equation is the equation \eqref{eqsimplextau}, 
where $\tau\in[0,1)$ is considered as a parameter. Using Hypothesis \ref{H3}, the fast equation admits the Perron-Frobenius vector $v(\tau)$ of $A(\tau)$ as an equilibrium which has a basin of attraction which is uniform.
Therefore, the conditions (SP2) are satisfied. Hence, $\eta(\tau,\varepsilon)$ is approximated by the slow curve $v(\tau)$ excepted on the set  $\bigcup_{k=0}^p[\tau_k,\tau_k+\nu]$, where $\nu$ is as small as we want.

Therefore, for any $\nu>0$, as small as we want, as $T\to \infty$,
the unique $T$-periodic solution $\eta^*(t,T)$ of  \eqref{eqtheta}   satisfies  
$$\eta^*(\tau,T)=v(\tau)+o(1)
\mbox{ uniformly on }[0,1]\setminus \bigcup_{k=1}^p[\tau_k,\tau_k+\nu],$$
where $\tau_0=0$ and $\tau_k$, $1\leq k\leq p$, are the discontinuity points of $A(\tau)$.
From this formula and $\theta^*(T\tau,T)=\eta^*(\tau,T)$ we deduce that
$$\theta^*(T\tau,T)=v(\tau)+o(1)
\mbox{ uniformly on }[0,1]\setminus \bigcup_{i=1}^p[\tau_k,\tau_k+\nu].$$ 
Since $\nu$ can be chosen as small as we want, as $T\to\infty$, using \eqref{Lambda=Lambda[rho]1}, 
$$\Lambda(T)=\int_0^{1}
\langle A(\tau)\theta^*(T\tau,T),{\bf 1}\rangle 
d\tau
=
\int_0^{1}
\langle A(\tau)v(\tau),{\bf 1}\rangle 
d\tau+o(1)=\overline{\lambda_{max}(A)}+o(1).
$$
For the details, we refer the reader to  \cite[Section 5.4]{BLSS}.
This proves \eqref{T=infini}.

\subsection{Proof of Proposition \ref{Lreducible}}
\label{ProofLreducible}
Since the columns of $L$ sum to 0, ${\bf 1}^\top L=0$, where ${\bf 1}=(1,\ldots,1)^\top$. Therefore $L^\top{\bf 1}=0$, which means that 0 is an eigenvalue of $L^\top$. Therefore, 0 is an eigenvalue of $L$.

For the rest of the proof, we begin by considering the case where $L$ is irreducible. Let $s=\max_{i}(-\ell_{ii})$ and let $B$ be the matrix $B=L+sI$. Note that $s$ is an eigenvalue of $B^\top$ and ${\bf 1}$ is a corresponding  eigenvector. The matrix $B^\top$ is non-negative and irreducible. By the Perron-Frobenius theorem, its spectral radius $\rho\left(B^\top\right)$ is a simple eigenvalue of 
$B^\top$ and it is the only eigenvalue of $B^\top$ which admits a positive eigenvector. 
This prove that $\rho\left(B^\top\right)=s$. Therefore, 
$\rho\left(B\right)=s$ and $s$ is a simple eigenvalue of $B$, admitting a positive eigenvector, which means that 0 is a simple eigenvalue of $L$ and it admits an eigenvector in $\Delta$. Moreover, all other eigenvalues of $B$ have modulus $<s$, which proves that all other eigenvalues of $L$ have real part $<0$. This proves the result when $L$ is irreducible and that $k=1$, in this case. 

 Assume now that $L$ is reducible.
By means of a permutation, we can put it into the following triangular block form

\begin{equation}
\label{red}
L=\left(
\begin{array}{cc}
M&0\\
N&P
\end{array}
\right)
\end{equation}
where $M$ and $P$ are square matrices.
If one of the matrices $M$ or $P$ is reducible, then by a permutation of its lines and columns, it can also be represented in a triangular block form similar to \eqref{red}. 
If one of the new diagonal blocks is reducible, then the process can be continued. Finally, by a suitable permutation of lines and columns, we can reduce $L$ to a triangular block form
\begin{equation}
\label{reds}
L=\left(
\begin{array}{cccc}
L_{11}&0&\cdots&0\\
L_{21}&L_{22}&\cdots&0
\\
\vdots&\vdots&\ddots&\vdots\\
L_{s1}&L_{s2}&\cdots&L_{ss}
\end{array}
\right),
\end{equation}
where 
the matrices $L_{11},\cdots L_{ss}$ are irreducible. Since $L$ is Metzler, the diagonal matrices $L_{ii}$ are also Metzler and the blocks $L_{ji}$, $j>i$, are non negative. Therefore, since the columns of $L$ sum to 0, the diagonal elements of the diagonal blocks are diagonally dominant in their columns, i.e. for any $p\in R_i$
\begin{equation}
\label{dominant}
|\ell_{pp}|=\sum_{j\neq p}\ell_{jp}\leq
\sum_{j\neq p,j\in R_i}\ell_{jp}=:\alpha_p(L_{ii}),
\end{equation}
where $\{1,\ldots,n\}=R_1\cup R_2\cup\ldots\cup R_s$ is 
the partition of the set $\{1,\ldots,n\}$ corresponding to the triangular block form \eqref{reds}.
Since the spectrum of $L$ satisfies
\begin{equation}
\label{spectre}
\sigma(L)=\sigma(L_{11})\cup\cdots\cup\sigma(L_{ss}),
\end{equation}
and $0\in\sigma(L)$, some $L_{ii}$ is singular. Let 
$$K=\left\{i\in\{1,\ldots,s\}:L_{ii}\mbox{ is singular}\right\},$$
and $k$, $1\leq k\leq s$, be the number of elements of $K$. 
Let us prove that for any $i\in K$, 
the blocks $L_{ji}$, that appear below $L_{ii}$ in the $i$th column of  \eqref{reds} are 0: for all $i<j\leq s$,  
\begin{equation}\label{Lji}
 L_{ji}=0.
\end{equation} 
Assume that some $L_{ji}\neq 0$ and let $p\in R_i$ and $q\in R_j$ such that 
$\ell_{qp}> 0$. Therefore, the inequality \eqref{dominant} is strict for $p$. It is well known that if an irreducible matrix, which is diagonally dominant, i.e. \eqref{dominant} holds, and if strict inequality holds in \eqref{dominant} for at least one $p$, then the matrix is nonsingular \cite[Theorem II]{Taussky} or \cite[Theorem 1.11]{Varga}.
Therefore, $L_{ii}$ is  non singular, a contradiction with $i\in K$.

Using \eqref{Lji}, we conclude that the columns of $L_{ii}$ sum to 0. Therefore, using the Perron-Frobenius theorem, as we did in the first part when we assumed that the matrix $L$ is irreducible, we can prove that
for each $i\in K$, 0 is a simple eigenvalue of $L_{ii}$, it admits an positive eigenvector, and all other eigenvalues of $L_{ii}$ are of negative real part. 

Now we consider $i\notin K$. The matrix $L_{ii}$ is not singular. Its eigenvalues belong to the Ger\v{s}gorin disks, centred at $\ell_{pp}$ and of radius 
$\alpha_p(L_{ii})$, where 
$\alpha_p(L_{ii})$ is the sum of the off diagonal elements of the $p$th column of $L_{ii}$ 
\cite[Theorem 1.1]{Varga}. Since $\ell_{pp}\leq 0$ and \eqref{dominant} holds, these disks are 
in the complex half-plane with a non positive real part.
Therefore, they can intersect the imaginary axis at most at the origin. But we know that 0 is not an eigenvalue of $L_{jj}$. Therefore, all eigenvalues of $L_{jj}$ are of negative real part. Using \eqref{spectre} we conclude that $k$ is the algebraic multiplicity of 0 and all other eigenvalues of $L$ are of negative real part.

Now we prove that $L$ admits $k$ linearly independent eigenvectors that can be chosen in the simplex $\Delta$.
As said before, if $i\in K$, $L_{ii}$ admits
an eigenvector $q\gg0$, such that $\|q\|_1=1$.
We obtain an eigenvector $p\in\Delta$ of $L$ by completing $q$ by 0 on the sets $R_j$, $j\neq i$. More precisely we define
$$p_j=\left\{
\begin{array}{lcl}
q_j&\mbox{if}&j\in R_i\\
0&\mbox{if}&j\notin R_i
\end{array}
\right.$$
Using \eqref{Lji}, we obtain
$Lp=0$. 
Therefore the geometric multiplicity of 0 is also equal to $k$ and the $k$ independent eigenvectors can be chosen in $\Delta$.

\begin{rem}
The sets $R_i$, $i\in K$ are known in Markov chains theory as the \emph{ergodic class} of states, while the sets $R_i$, $i\notin K$ correspond to the \emph{transient class} \cite[Page 695]{MeyerBook}. The positive  eigenvectors $p\in\Delta$ correspond to the \emph{ergodic measures} of the Markov chain. In their extension of the Perron-Frobenius theorem to reducible non negative matrices, Berman and Plemmons \cite[Theorem 3.20]{Berman} called the sets $R_i$, $i\in K$, \emph{final class}, see \cite[Definition 3.8]{Berman}.
\end{rem}

\subsection{Proof of Theorem \ref{Prop15}}
\label{ProofProp15}

The proof follows the same steps as the proof of \cite[Eq. (15)]{BLSS}. Indeed,
the result in \cite{BLSS} only uses
Tikhonov's theorem and the fact that 
the Perron-Frobenius vector $p(\tau)$ of $L(\tau)$ is globally asymptotically stable for  \eqref{eqsimplextauL}. Note that, in the case where $L(\tau)$ is irreducible for all $\tau\in[0,1]$, we have $p(\tau)\gg 0$, and its global asymptotic stability in the simplex $\Delta$ is guaranteed, see \cite[Proposition 24]{BLSS}. In the more general case where only the average matrix $\overline{L}$ is assumed to be irreducible, we need to introduce Hypothesis \ref{H4}.

 Using the decomposition $A(\tau)=R(\tau)+mL(\tau)$, the equation \eqref{eqtheta} on the simplex $\Delta$ is
 \begin{equation}\label{eqt0}
\frac{d\theta}{dt}=R(t/T)\theta+mL(t/T)\theta-\langle
R(t/T)\theta,{\bf 1}\rangle\theta-
m\langle L(t/T)\theta,{\bf 1}\rangle\theta,
\end{equation}
Since the columns of $L(\tau)$ sum to 0, 
$\langle L(\tau)\theta,{\bf 1}\rangle=0$. Therefore, using the variables $\tau=t/T$ and $\eta(\tau)=\theta(T\tau)$, this equation is written
\begin{equation}\label{ztau0}
\frac{1}{m}\frac{d\eta}{d\tau}=TL(\tau)\eta
+\frac{1}{m}\left[TR(\tau)\eta-T\langle
R(\tau)\eta,{\bf 1}\rangle\eta\right].
\end{equation}

Theorem \ref{ThmThik}, in Appendix \ref{SingPert}, applies to the equation \eqref{ztau0}, which can be written, noting $\varepsilon=1/m$, as
\begin{equation*}
\varepsilon\frac{d\eta}{d\tau}=TL(\tau)\eta
+\varepsilon\left[TR(\tau)\eta-T\langle
R(\tau)\eta,{\bf 1}\rangle\eta\right].
\end{equation*}
Using the fast time $s=m\tau$, the fast equation is the equation \eqref{eqsimplextauL},
where $\tau$ is considered as a parameter. 
 Using Hypothesis \ref{H4}, the Perron-Frobenius vector $p(\tau)$ of $L(\tau)$, is GAS in $\Delta$  for \eqref{eqsimplextauL}, see Lemma \ref{lemma10}. Therefore, the conditions (SP2) are satisfied. Hence, $\eta(\tau,\varepsilon)$ is approximated by the slow curve $p(\tau)$ excepted on the set  $\bigcup_{k=0}^p[\tau_k,\tau_k+\nu]$, where $\nu$ is as small as we want.

Therefore, for any $\nu>0$, as small as we want, as $T\to \infty$, the unique $T$-periodic solution $\theta^*(t,T)$ of \eqref{eqt0}  satisfies  
$$\theta^*(T\tau,m)=p(\tau)+o(1)
\mbox{ uniformly on }[0,1]\setminus \bigcup_{i=1}^p[\tau_k,\tau_k+\nu],$$
where $\tau_0=0$ and $\tau_k$, $1\leq k\leq p$, are the discontinuity points of $A(\tau)=R(\tau)+mL(\tau)$.  
From \eqref{Lambda=Lambda[rho]1}, 
$\Lambda(m,T)=\int_0^{1}
\langle R(\tau)\theta^*(T\tau,T),{\bf 1}\rangle 
d\tau$.
Since $\nu$ can be chosen as small as we want, as $m\to\infty$, 
$$
\begin{array}{l}
\int_0^{1}\langle R(\tau)\theta^*(T\tau,T),{\bf 1}\rangle d\tau
=
\int_0^{1}\langle R(\tau)p(\tau),{\bf 1}\rangle d\tau+o(1)
=\sum_{i=1}^n\overline{p_ir_i}+o(1).
\end{array}
$$
For the details, we refer the reader to  \cite[Section 5.6]{BLSS}.
This proves \eqref{m=infini}.

\subsection{Proof of Proposition \ref{PropES}}
\label{ProofPropES}
The monodromy matrix of \eqref{eq47} is  given by
$
X(m,T) = e^{\frac{T}{3} A_3}e^{\frac{T}{3} A_2}e^{\frac{T}{3} A_1}.
$
We have
\[
e^{t A_1} = \begin{pmatrix}
e^{a t} & 0 & 0\\
0 & e^{b t}F & e^{bt}G\\
0 & e^{bt}G & e^{bt}F
\end{pmatrix},
~
e^{t A_2} = \begin{pmatrix}
e^{b t}F & 0 & e^{bt}G\\
0 & e^{a t} & 0 \\
e^{bt}G & 0& e^{bt}F
\end{pmatrix},
~
e^{t A_3} = \begin{pmatrix}
e^{b t}F & e^{bt}G& 0\\
e^{bt}G & e^{bt}F& 0\\
 0 & 0& e^{a t}\\
\end{pmatrix},
\]
where 
\begin{equation}
\label{FG}
F= \frac{1}{2}(1 + e^{- 2 m t}),\quad
G = \frac{1}{2}(1 - e^{- 2 m t}).
\end{equation}
A straightforward calculation shows that
$$X(m,T)=e^{\frac{a+2b}{3}T}Y(m,T),$$ 
where
\[
Y(m,T) =  
\begin{pmatrix}
F^2& FG+\beta FG^2 & G^2+\beta F^2G\\
FG & F^2+\beta G^3 &
FG+\beta FG^2 \\
\alpha G & FG & F^2 
\end{pmatrix}.
\]
where 
$\alpha=e^{\frac{a-b}{3}T}$,
$\beta=e^{\frac{b-a}{3}T},$
and $F$ and $G$, defined by \eqref{FG}, are evaluated for $t=T/3$. 
The Perron root of $X(m,T)$ is given by
$$\mu(m,T)=e^{\frac{a+2b}{3}T}\lambda_{max}(Y(m,T)).$$
Therefore, the growth rate is given by
$$
\Lambda(m,T)=\frac{1}{T}\ln(\mu(m,T))=
\frac{a+2b}{3}+\frac{1}{T}\ln(\lambda_{max}(Y(m,T))).
$$
From the expressions \eqref{FG} of $F$ and $G$ we deduce that
\begin{equation}\label{Y(m,T)}
Y(m,T)=B(T)+e^{-2mT/3}Z(m,T) 
\end{equation}
where
\[
B(T)= \frac{1}{4}\begin{pmatrix}
1& 1+\frac{\beta}{2} & 1+\frac{\beta}{2} \\
1 & 1+\frac{\beta}{2} &1+\frac{\beta}{2} \\
2\alpha & 1 & 1 
\end{pmatrix},
\]
and $Z(m,T)$ is a bounded matrix.
Therefore, using the continuity of the spectral abscissa, for each fixed $T>0$, as $m\to\infty$, 
$$\lambda_{max}(Y(m,T))=\lambda_{max}(B(T))+o(1).$$
We can check that 
\[
\lambda_{\max}(B(T)) = \frac{3/2+
\beta/4+(1/4)\sqrt{36+\beta^2+12\beta+32\alpha}}{4},
\]
Therefore, as $m\to\infty$
$$\Lambda(m,T)=\frac{a+2b}{3}+\frac{1}{T}\ln
\left(\frac{3/2+
\beta/4+(1/4)\sqrt{36+\beta^2+12\beta+32\alpha}}{4}\right)+o(1).
$$
Using again \eqref{Y(m,T)} and the continuity of the spectral abscissa, for each fixed $m>0$, as $T\to\infty$, 
$$\lambda_{max}(Y(m,T))=\lambda_{max}(B(T))+o(1).$$
Since $a>b$, as $T\to\infty$, we have
\[
\lambda_{\max}(B(T)) = e^{\frac{a-b}{6}T}\left(\sqrt{2}/4+O\left(e^{\frac{b-a}{6}T}\right)\right)
\]
Therefore, as $T\to\infty$, 
$$\Lambda(m,T)=
\frac{a+2b}{3}+\frac{a-b}{6}+o(1)=
\frac{a + b}{2} + o(1).$$

\section{Singular perturbations}\label{SingPert}
Consider the singularly perturbed differential equation
\begin{equation}
\label{E1}
\varepsilon\frac{d\eta}{d\tau}=f(\tau,\eta,\varepsilon),
\end{equation}
where $f:[0,1]\times\mathbb{R}^n\times(0,\infty)\to\mathbb{R}^n$ is a function, verifying the conditions $(SP1)$  and $(SP2)$ that will be specified below.

When $\varepsilon\to 0$, \eqref{E1} is a singularly perturbed equation with  $n$ fast variables $\eta$ and one slow variable $\tau$. 
Using the fast time $s=\tau/\varepsilon$, this equation can be rewritten
\begin{equation}
\label{E2}
\begin{array}{rl}
\frac{d\eta}{ds}&=f(\tau,\eta,\varepsilon),\\
\frac{d\tau}{ds}&=\varepsilon.
\end{array}
\end{equation}
The slow curve of \eqref{E2} is given by $\eta=\xi(\tau)$, where $\xi(\tau)$ is the equilibrium of the fast equation  
\begin{equation}
\label{E3}
\frac{d\eta}{ds}=f(\tau,\eta,0),
\end{equation}
obtained by letting 
$\varepsilon=0$ in \eqref{E2}. Therefore, in the fast equation \eqref{E3}, $\tau$ is considered as a constant parameter. We make the following assumptions. 

\begin{description}
\item[(SP1)] There is a finite set 
$D=\{\tau_k, 1\leq k\leq p:  0< \tau_1<\cdots<\tau_p<1\}$, such that $f$ is continuous on $\left([0,1]\setminus D\right)\times\mathbb{R}^n\times(0,\infty)$, differentiable with respect of $\eta$, and has right and left limits at the discontinuity points $\tau_k\in(0,1)$, $k=1,\ldots,p$.
\item[(SP2)] 
The equilibrium $\eta=\xi(\tau)$ of the fast equation \eqref{E3} is asymptotically stable with a basin of attraction which is uniform with respect on $\tau\in[0,1)$. This means that:\\
{\it(SP2.1)} For each $\tau\in[0,1)$, $\xi(\tau)$ is an asymptotically stable equilibrium of \eqref{E3}.\\
{\it(SP2.2)}  
There exists $\delta>0$ such that for each $\tau\in[0,1)$, the $\delta$-radius ball, centred at $\xi(\tau)$, is included in the basin of attraction of $\xi(\tau)$.\\
{\it(SP2.3)} At the points $\tau_k$ of discontinuity, the basin of attraction of $\xi(\tau_{k})$ contains the limit at right 
$\xi(\tau_{k}-0):=\lim_{\tau\to \tau_k,\tau<\tau_k}\xi(\tau).$
\end{description}

Using the Tikhonov's theorem on singular perturbations we have the following result.

\begin{theorem}\label{ThmThik}
Assume that the conditions (SP1) and (SP2) are satisfied. Assume that $\eta_0$ belongs to the basin of attraction of $\xi(0)$.
Let $\nu>0$. For $\varepsilon$ small enough, the solution $\eta(\tau,\varepsilon)$ of \eqref{E1} with initial condition $\eta(0,\varepsilon)=\eta_0$, is defined on $[0,1]$ and, as  $\varepsilon\to 0$,
$$
\begin{array}{ll}
\eta(\tau,\varepsilon)=\xi(\tau)+o(1)
&\mbox{ uniformly on } [0,1]\setminus \bigcup_{k=0}^p[\tau_k,\tau_k+\nu],
\end{array} 
$$
where $\tau_0=0$ and $\tau_k$, $1\leq k\leq p$, are the discontinuity points of $f$.
\end{theorem}

\begin{proof}
The result is a particular case of \cite[Proposition 27]{BLSS}.
\end{proof}

\section*{Acknowledgements} We warmly thank the handling editor and the anonymous reviewer whose constructive remarks have produced an improved version of our article.

\begin{itemize}
\item Conflict of interest/Competing interests: none
\item Availability of data and materials: Data sharing not applicable to this article as no datasets were generated or analysed during the current study.
\end{itemize}

{}


\end{document}